\documentclass[10pt,letterpaper]{amsart}
\numberwithin{equation}{section}
\numberwithin{figure}{section}

\usepackage{amssymb}
\usepackage{mathtools}
\usepackage[headings]{fullpage}
\usepackage{amsmath}
\usepackage{graphicx}
\usepackage{thmtools}

\usepackage{color}
\usepackage[normalem]{ulem}
\usepackage[hidelinks]{hyperref}
\usepackage{dsfont}
\hypersetup{
    colorlinks=true,
    linkcolor=blue,
    urlcolor=red,
    citecolor=Green
    } 
\DeclareTextFontCommand{\emph}{\color{blue}\em}
\usepackage{enumitem}
\usepackage{tikz}
\usepackage{tikz-cd}

\usepackage{xparse}

\newcount\tableauRow
\newcount\tableauCol

\usetikzlibrary {positioning}
\definecolor {processblue}{cmyk}{0.96,0,0,0}

\makeatletter
\def\theenumi{\@alph\c@enumi}
\makeatother

\usepackage{units}
\usepackage{ytableau}

\theoremstyle{plain}
\newtheorem{theorem}[equation]{Theorem}
\newtheorem{lemma}[equation]{Lemma}
\newtheorem{corollary}[equation]{Corollary}
\newtheorem{proposition}[equation]{Proposition}
\theoremstyle{definition}

\newtheorem{remark}[equation]{Remark}

\newtheorem{example}[equation]{Example}

\newtheorem{definition}[equation]{Definition}

\newtheorem{notation}[equation]{Notation}

\newtheorem{discussion}[equation]{Discussion}

\newtheorem{observation}[equation]{Observation}

\newtheorem{construction}[equation]{Construction}

\newcommand{\sign}[1]{\mathrm{sgn(#1)}}

\newcommand{\rcf}[2]{\mathcal{R}_{\scriptscriptstyle #1,#2}}

\newcommand{\reval}[1]{\mathfrak{R}_{-,#1}}

\newcommand{\sorting}[1]{\operatorname{sort}(#1)}

\newcommand{\bare}[1]{\overline{e}_{#1}}

\usepackage{xparse}
\usepackage[dvipsnames]{xcolor}
\usepackage{verbatim}

\DeclareMathOperator{\row}{rind}  

\DeclareMathOperator{\Gar}{Garc}

\usepackage[backend=biber, style=alphabetic, maxbibnames=999, maxalphanames=999,doi=true,
  url=false]{biblatex}
\DeclareLabelalphaTemplate{%
  \labelelement{\field[final]{shorthand}}%
  \labelelement{\field[uppercase,strwidth=1,strside=left]{labelname}}%
  \labelelement{\field[strwidth=2,strside=right]{year}\field{extrayear}}%
}
\AtBeginDocument{%
}

\title[Noniterative Straightening for Skew Schur Modules]{A Non-Iterative Straightening Algorithm and Orthogonality for Skew Schur Modules}

\author{Reuven Hodges}
\address{Department of Mathematics, University of Kansas, Lawrence, KS, 66045, USA}
\email{rmhodges@ku.edu}
\author{Hanzhang Yin}
\address{Department of Mathematics, University of Kansas, Lawrence, KS, 66045, USA}
\email{hanyin@ku.edu}

\date{\today}

\begin{document}

\begin{abstract}
We generalize Fulton's determinantal construction of Schur modules to the skew setting, providing an explicit and functorial presentation using only elementary linear algebra and determinantal identities, in parallel with the partition case \cite{97FW}. Building on the non-iterative straightening formula of the first author for partition shapes \cite{24RH}, we develop a non-iterative straightening algorithm for skew Schur modules that expresses arbitrary elements in a new D-basis with an explicit closed coefficient formula. We then show that this D-basis is the result of applying Gram-Schmidt orthogonalization to the semistandard tableau basis, which identifies a natural inner product on the skew Schur module and recasts straightening as an orthogonal projection.
\end{abstract}

\maketitle

\section{Introduction}

The classical straightening algorithm is a foundational normal-form procedure, furnishing canonical bases for determinantal ideals and invariant rings by systematically applying a finite repertoire of determinantal identities, a process now formalized and generalized by the theory of Gr\"{o}bner and SAGBI bases \cite{Hodge1943, SturmfelsAIT, GroebnerDeterminantal}. These determinantal identities, collectively known as the determinantal calculus~\cite{HodgePedoeVol2}, are grounded in the multilinear and alternating properties of the determinant and serve to codify the intricate relations and syzygies among minors.

This same algebraic machinery provides the modern foundation for constructing fundamental objects in the representation theory of the symmetric group $\mathfrak{S}_n$ and the general linear group $GL_n$. Here, the straightening algorithm is intuitively realized through the combinatorial language of Young tableaux, where it translates into a procedure for reducing an arbitrary integer filling of a Young diagram to a unique linear combination of semistandard Young tableaux (SSYT) \cite{97FW, 03WJ, 01BS}. In this setting, the abstract determinantal identities re-emerge as concrete combinatorial moves, namely, column alternation and the Garnir relations. The representation theory of $\mathfrak{S}_n$ and its Specht module theory can be traced through work of Young, Specht, and Garnir on tableaux, symmetrizers, and the Garnir relations \cite{Y01,S35,50GA}. On the $GL_n$ side, Schur and Weyl established the classical representation-theoretic framework, including Schur-Weyl duality and the polynomial representations, while Akin-Buchsbaum-Weyman gave a functorial, characteristic-free treatment that extends naturally to skew modules~\cite{Schur1901,W39,ABW82}. Taken together, these constructions formalize straightening as the mechanism behind standard bases in the representation theory of $\mathfrak{S}_n$ and $GL_n$, respectively.

This algebraic framework finds its geometric realization in standard monomial theory (SMT). The SMT program was initiated by Seshadri in the 1970s and, together with Lakshmibai and Musili, developed to generalize Hodge's standard monomial basis for the coordinate ring of the Grassmannian to the coordinate rings of partial flag varieties and their Schubert subvarieties \cite{SeshadriSMT,LMS}. The standard monomials of SMT correspond directly to SSYT, and the classical straightening relations on SSYT are precisely the rules needed to reduce arbitrary monomials to the standard monomial basis. Littelmann's path model proved the general SMT conjectures for all reductive groups \cite{Littelmann}. For broader context, the expository program of Kung-Rota and Rota-Stein situates straightening, via symbolic methods, bitableaux, and Young symmetrizers, as the normal-form mechanism in classical invariant theory, clarifying its structural role \cite{KungRota1984,PNASSteinRota}.

Despite its theoretical elegance, the iterative nature of the classical straightening algorithm presents a significant computational obstacle. The process, which involves repeatedly finding and resolving violations of the semistandard condition, often leads to an intermediate explosion in the size of the expressions, resulting in unstable and inefficient computations \cite{White1990}. From a computational standpoint, straightening appears throughout Sturmfels' account as the normalization routine for bracket and minor algebras and as the reduction step behind Gr\"{o}bner and SAGBI workflows in invariant theory \cite{SturmfelsAIT}. These computational aspects motivate the search for direct, non-iterative formulas that bypass the stepwise reduction process.

Recently, the first author provided such a non-iterative formula for fillings of partition shapes \cite{24RH}. The present paper extends this construction to the more general setting of skew Schur modules and demonstrates that the resulting basis has a natural geometric interpretation. We establish a determinantal framework and define a family of evaluation functionals that directly compute expansion coefficients. This approach yields a new basis that is not merely a combinatorial artifact; we demonstrate that it is precisely the basis obtained by applying the Gram-Schmidt process to the basis of SSYT with respect to a natural sesquilinear form. In this light, the straightening expansion in this new basis is revealed to be an orthogonal projection. This geometric perspective provides a unified and structurally natural explanation for the unitriangularity of the basis transformation and explicitly identifies the expansion coefficients as orthogonal coordinates. An orthonormal D-basis supplies canonical coordinates and orthogonal projections in the skew Schur module, providing fine-grained structural control over expansions and filtrations.

\subsection{Main Results} To state our main results, we introduce some minimal notation. Let $\lambda/\mu$ be a skew partition. We denote the set of integer fillings of this shape with a fixed content $z$ by $F(\lambda/\mu, z)$ and the subset of semistandard Young tableaux (SSYT) by $\operatorname{SSYT}(\lambda/\mu, z)$. Our object of study is the skew Schur module $E^{\lambda/\mu}$, which is constructed functorially from an $R$-module $E$. For any filling $F$, we denote its corresponding element in the module by $\bare{F}$. Definitions of the mathematical objects in this section appear in the preliminaries; notation follows that section throughout.

Our work begins by establishing a determinantal, linear-algebraic construction of the skew Schur module. This approach avoids the heavier algebraic machinery of existing general constructions \cite{03WJ} and instead parallels the accessible framework used by Fulton for Schur modules associated to partition shapes \cite{97FW}. A key feature of our construction is an explicit $R$-module homomorphism $\Psi: E^{\lambda/\mu} \to R[Z]$ from the module to a polynomial ring. Here, $R[Z]$ is the ring $R[Z_{i,j}]$, where $i$ ranges over the row indices present in the skew diagram $\lambda/\mu$ and $j$ ranges over the alphabet $[m]$. This homomorphism allows us to provide a new determinantal proof that the submodule corresponding to a fixed content $z$ has a basis indexed by the set of SSYT $\operatorname{SSYT}(\lambda/\mu, z)$, relative to the chosen basis of $E$. Moreover, this homomorphism underlies the proof of the non-iterative straightening formula in Theorem \ref{thm:manResult1}.

Our first main result is the central combinatorial achievement: a non-iterative straightening formula. To state it, we fix a total order (the reading-word order, $\prec$) on the SSYT, $S_n \prec \dots \prec S_2 \prec S_1$. From this, we recursively define a new basis, the D-basis $\{\bare{D_i}\}$, as a specific unitriangular transformation of the SSYT basis $\{\bare{S_i}\}$. The coefficients in the formula are given by rearrangement coefficients $\rcf{F}{S}$, which are the signed count of column permutations transforming the multiset of entries in each column of $F$ to that of $S$.

\begin{restatable}[Non-Iterative Straightening Formula]{theorem}{thmA}
\label{thm:manResult1}
For any filling $F \in F(\lambda/\mu, z)$, its expansion in the D-basis is given by:
\[ \bare{F} = \sum_{S_j \in \operatorname{SSYT}(\lambda/\mu, z)} \rcf{F}{S_j} \bare{D_j}. \]
\end{restatable}

Our second main result reveals that this combinatorial framework has a natural geometric interpretation. This requires specific assumptions on the base ring $R$: we assume it is a commutative ring equipped with an involutive automorphism $r \mapsto r^*$ such that the subring of fixed points is ordered, $rr^* \ge 0$ for all $r \in R$, and $rr^*=0$ only if $r=0$. We define a sesquilinear form $\langle \cdot, \cdot \rangle$ on the skew Schur module by specifying its values on the SSYT and D-bases as $\langle \bare{S_i}, \bare{D_j} \rangle := R_{S_i,S_j}$ and extending sesquilinearly.

\begin{restatable}[Geometric Interpretation and Orthogonality]{theorem}{thmB}
\label{thm:B}
The D-basis is the result of applying the Gram-Schmidt orthogonalization process to the SSYT basis $\{\bare{S_i}\}$, taken in the fixed reading-word order, under the form $\langle \cdot, \cdot \rangle$. Moreover, the D-basis is orthonormal:
\[ \langle \bare{D_i}, \bare{D_j} \rangle = \delta_{ij}. \]
Consequently, the skew Schur module $E^{\lambda/\mu}$ is a inner product $R$-module, and the straightening formula of Theorem~\ref{thm:manResult1} is an orthogonal projection.
\end{restatable}

\subsection{Practical Applications} The demand for efficient straightening algorithms is partly driven by their role as a core subroutine in computational representation theory. This is apparent in pipelines such as Young flattenings, which require stable, large-scale manipulations in a Schur-module basis, and in routines for problems related to Foulkes' conjecture that rely on repeated symmetrization \cite{YoungFlatteningsArxiv, IkenmeyerFoulkes}. Young flattenings are equivariant linear maps between Schur modules that yield determinantal equations for secant varieties and certify border-rank lower bounds \cite{OL13}. Non-iterative straightening has been applied in this context; the partition-case formula of \cite{24RH} was successfully implemented as part of the work in \cite{YoungFlatteningsArxiv}, yielding substantial performance improvements over classical iterative methods. In symmetrization computations underlying Foulkes-type problems, the task is to determine dimensions of distinguished subspaces in the quotient~\cite{IkenmeyerFoulkes}. These dimensions can be efficiently computed by comparing rearrangement-coefficient coordinate vectors in the D-basis, since subspace dimension is invariant under change of basis and the coordinates are computable without constructing the D-basis.

\subsection{Outline} The remainder of the paper is organized as follows. Section~\ref{sec:prelim} develops our determinantal construction of the skew Schur module and introduces the key homomorphism into a polynomial ring that underpins our non-iterative approach. Section~\ref{sec:rearrangement} defines the rearrangement coefficients, establishes their key structural properties, constructs the linear functionals used to compute them, and defines the $D$-basis. Section~\ref{sec:noniterativestraighten} synthesizes these tools to prove the non-iterative straightening formula in Theorem~\ref{thm:manResult1}. Finally, Section~\ref{sec:ortho} introduces the inner product on the skew Schur module and proves that the D-basis is the orthonormal basis obtained via the Gram-Schmidt process, yielding Theorem~\ref{thm:B}, and establishing the geometric interpretation of our straightening algorithm as an orthogonal projection.

\section{Acknowledgements}
We thank Darij Grinberg for posing the question that motivated this project and Jeremy Martin for comments and suggestions that improved the exposition. The first author thanks Andreas Karrenbauer for suggesting that the partition-case construction of the D-basis resembles Gram--Schmidt orthogonalization, which motivated our orthogonality-based interpretation of the D-basis. This project benefited from computations performed with SageMath.

\section{Preliminaries} \label{sec:prelim}
In this section, we give a determinantal, linear-algebraic construction of the skew Schur module that avoids the heavier algebraic machinery found in \cite{03WJ} and parallels the construction of Schur modules in \cite{97FW}. A key ingredient of our approach is an explicit module homomorphism from the skew Schur module into a polynomial ring, which underpins both the proof of the non-iterative straightening algorithm and the analysis of its coefficients.

\subsection{Tensor and exterior powers}
Let $R$ be a commutative ring with $1$ and let $E$ be a finite free $R$-module. For $r\ge 1$, the $r$-fold tensor power is $E^{\otimes r}=E\otimes_R\cdots\otimes_R E$ (with $r$ factors).  Let $R_0=\mathbb{Z}\cdot 1_R\subset R$ denote the prime subring. For $r\ge 0$, the $r$-th exterior power $\bigwedge^r E$ is the quotient of $E^{\otimes r}$ by the $R$-submodule generated by elementary tensors $e_1\otimes\cdots\otimes e_r$ with $e_i=e_j$ for some $i\ne j$, and we write $e_1\wedge\cdots\wedge e_r$ for the image of $e_1\otimes\cdots\otimes e_r$. This construction is $R$-multilinear and alternating in its arguments, so $v\wedge v=0$.

\subsection{Skew partitions, fillings, and SSYT}

The combinatorial objects in this section provide a framework for constructing bases of skew Schur modules.

A \emph{partition} is a finite sequence of positive integers $\lambda=(\lambda_1,\ldots,\lambda_k)$ with $\lambda_1\ge\cdots\ge\lambda_k>0$. Its \emph{length} is $\ell(\lambda)=k$ and its \emph{size} is $|\lambda|=\sum_{i=1}^k \lambda_i$. Identify $\lambda$ with its \emph{Young diagram}, a collection of left-justified boxes with $ \lambda_i $ boxes in row $ i $; write $(r,c)$ for the box in row $r$ and column $c$. For a partition $ \lambda $, the \emph{conjugate partition} $ \lambda' = (\lambda'_1, \ldots, \lambda'_{\lambda_1}) $ is defined by letting $ \lambda'_j $ be the number of boxes in column $ j $ of the Young diagram of $ \lambda $. Visually, the Young diagram of $\lambda'$ is the reflection of the Young diagram of $\lambda$ across its main diagonal.

Let $\lambda$ and $\mu$ be partitions with $\mu_i\le \lambda_i$ for all $i$. Then $\lambda/\mu$ is a \emph{skew partition} with its \emph{skew Young diagram} equal to the set of boxes
\[
D(\lambda/\mu) = \{(r,c)\in\mathbb{Z}_{>0}^2 \mid 1\le r\le \ell(\lambda),\ \mu_r< c \le \lambda_r\},
\]
and size $|\lambda/\mu|=|\lambda|-|\mu|$.

We adopt the standard convention of using the same notation $\lambda/\mu$ for the skew partition and its skew Young diagram. For $c\in[\lambda_1]$, the $c$-th column is $\mathrm{Col}_c(\lambda/\mu)=\{(r,c)\in\lambda/\mu\}$ with height $\ell'_c=|\mathrm{Col}_c(\lambda/\mu)|$. Then $\ell'_c=\lambda'_c-\mu'_c$.

Fix $m\in\mathbb{Z}_{>0}$ and write $[m]=\{1,\ldots,m\}$. A \emph{filling} of shape $\lambda/\mu$ is a function $F:\lambda/\mu\to [m]$, and $F[r,c]$ denotes the entry in cell $(r,c)$. The \emph{content} of $F$ is $z=(z_1,\ldots,z_m)\in\mathbb{Z}_{\ge 0}^m$, where $z_i$ counts the number of entries equal to $i$. A \emph{semistandard young tableaux (SSYT)} of shape $\lambda/\mu$ is a filling with entries in rows weakly increasing left to right and entries in columns strictly increasing top to bottom. Write $F(\lambda/\mu,z)$ for the set of fillings with content $z$ and $\text{SSYT}(\lambda/\mu,z)$ for the set of SSYT with content $z$.

\begin{example}
\ytableausetup{notabloids}
\ytableausetup{mathmode,boxsize=1.25em,centertableaux}
Let $\lambda=(5,4,3,2)$ and $\mu=(3,2)$. Then $\lambda'=(4,4,3,2,1)$ and
  \[
    \ytableausetup{mathmode,boxsize=1em,centertableaux}
    \lambda =
    \begin{ytableau}
      \; & \; & \; & \; & \; \\
      \; & \; & \; & \; \\
      \; & \; & \; \\
      \; & \;
    \end{ytableau}
    \qquad
    \mu =
    \begin{ytableau}
      \; & \; & \; \\
      \; & \;
    \end{ytableau}
    \qquad
    \lambda/\mu =
    \begin{ytableau}
      \none & \none & \none & \; & \; \\
      \none & \none & \; & \;  \\
      \; & \; & \; \\
      \; & \;
    \end{ytableau}
  \]
The size of the skew partition is $|\lambda/\mu| = 14-5=9$. For $\lambda=(3,2)$ and $\mu=(1)$, the skew shape $\lambda/\mu$ has cells $\{(1,2), (1,3), (2,1), (2,2)\}$ and
  \[
    \ytableausetup{mathmode,boxsize=1em,centertableaux}
    \begin{ytableau}
      \none & 1 & 2 \\
      3 & 2
  \end{ytableau}
  \qquad \text{ and }\qquad
  \begin{ytableau}
      \none & 1 & 2 \\
      2 & 3
  \end{ytableau},
  \]
are examples, respectively, of a filling and an SSYT of shape $\lambda/\mu$.
\end{example}

\subsection{Garnir action on fillings} \label{subsec:garnirdef}

To define the skew Schur module, we must first define the action of a certain subset of the symmetric group on fillings. These actions lead to the Garnir relations, which were originally introduced by Garnir in \cite{50GA} to provide a straightening algorithm for standard Young tableaux. We employ a version of this machinery adapted to skew shapes.

Let $a$ and $b$ be positive integers and $\mathfrak{S}_{a+b}$ be the symmetric group on $a+b$ elements. Throughout the paper, we will use one-line notation for permutations.
Let
\[
\mathfrak{S}_{a+b}^{a,b} = \{\sigma \in \mathfrak{S}_{a+b} \mid \sigma(1) < \sigma(2) < \dots < \sigma(a), \text{ and } \sigma(a+1) < \dots < \sigma(a+b)\}
\]
be the set of $(a,b)$-shuffles. This is the distinguished set of minimal length representatives for the left cosets of the Young subgroup $\mathfrak{S}_a \times \mathfrak{S}_b$ in the symmetric group $\mathfrak{S}_{a+b}$.

Now, fix a skew partition $\lambda/\mu$ and two distinct columns $c_1, c_2 \in [\lambda_1]$ such that $c_1 < c_2$. Choose positive integers $a \leq \ell'_{c_1}$ and $b \leq \ell'_{c_2}$. Given a filling $F$ of shape $\lambda/\mu$, we define two sets of coordinates: $\Gar^{a,\_}_{c_1,\_}$, corresponding to the bottom $a$ cells of column $c_1$, and $\Gar^{\_,b}_{\_, c_2}$, corresponding to the top $b$ cells of column $c_2$. Formally,
\[
\Gar^{a,\_}_{c_1,\_} = \{(\lambda'_{c_1}-a+1, c_1), \dots, (\lambda'_{c_1}, c_1)\}, \quad
\Gar^{\_,b}_{\_, c_2} = \{(\mu'_{c_2}+1, c_2), \dots, (\mu'_{c_2}+b, c_2)\}.
\]
Their union
\[
\Gar_{c_1,c_2}^{a,b} = \Gar^{a,\_}_{c_1,\_} \cup \Gar^{\_,b}_{\_, c_2}
\]
is the set of coordinates whose entries in $F$ will be permuted by $(a,b)$-shuffles. Define the bijection
\[
\eta : \Gar_{c_1,c_2}^{a,b} \to \{1, 2, \dots, a+b\}
\]
that enumerates these coordinates by first listing the cells in the $c_1$ segment from top to bottom, followed by the cells in the $c_2$ segment from top to bottom. Specifically, $\eta$ maps
\[
(\lambda'_{c_1}-a+1,c_1)\mapsto 1,
\dots,
(\lambda'_{c_1},c_1)\mapsto a,
(\mu'_{c_2}+1,c_2)\mapsto a+1,
\dots,
(\mu'_{c_2}+b,c_2)\mapsto a+b.
\]

\begin{definition} \label{def:permutation_filling}
For each permutation $\pi \in \mathfrak{S}_{a+b}^{a,b}$, we define a new filling $\pi(F^{a, b}_{c_1, c_2})$ obtained from $F$ by keeping the entries for coordinates outside $\Gar_{c_1,c_2}^{a,b}$ fixed and rearranging the entries for coordinates inside according to $\pi$. The action is given by the formula
\[
\pi\!\left(F^{a,b}_{c_1,c_2}\right)[x,y]=
\begin{cases}
F\!\left[\eta^{-1}\!\big(\pi(\eta(x,y))\big)\right], & (x,y)\in \Gar_{c_1,c_2}^{a,b},\\[4pt]
F[x,y], & (x,y)\notin \Gar_{c_1,c_2}^{a,b}.
\end{cases}
\]
\end{definition}
For notational simplicity, we let $F_{\pi}$ denote the new filling $\pi(F^{a, b}_{c_1, c_2})$. The parameters $a, b, c_1, c_2$ will always be clear from the surrounding context.

\begin{example}
Let $\lambda = (3,2)$ and $\mu = (1)$. Consider the filling $F$ of the skew shape $\lambda/\mu$ given by
\[
\ytableausetup{mathmode,boxsize=1em,centertableaux}
F = \begin{ytableau}
\none & 2 & 1 \\
3 & 1
\end{ytableau}.
\]
The conjugate partition is $\lambda' = (2,2,1)$. For column indices, we choose $c_1 = 1$ and $c_2 = 2$. The corresponding column heights in the skew shape are
\[
\ell'_1 = \lambda'_{1} - \mu'_{1} = 2 - 1 = 1, \quad \text{and} \quad \ell'_2 = \lambda'_{2} - \mu'_{2} = 2 - 0 = 2.
\]
Let us consider the case where we take the maximal number of cells from each column, so we set $a = \ell'_1 = 1$ and $b = \ell'_2 = 2$. The set of $(1,2)$-shuffles, where we represent the permutations by their one-line notation, is
\[
\mathfrak{S}_{1+2}^{1,2} = \{\sigma \in \mathfrak{S}_3 \mid \sigma(2) < \sigma(3)\} = \{123, 213, 312\}.
\]
Let $\pi = 213 \in \mathfrak{S}_{1+2}^{1,2}$. We now construct the new filling $F_{\pi}$. The relevant coordinate sets are
\[
\Gar^{1,\_}_{1,\_} = \{(2,1)\}, \quad \Gar^{\_,2}_{\_,2} = \{(1,2), (2,2)\}, \quad \text{and} \quad \Gar^{1,2}_{1,2} = \{(2,1), (1,2), (2,2)\}.
\]
The bijection $\eta : \Gar^{1,2}_{1,2} \to \{1,2,3\}$ maps $(2,1) \mapsto 1$, $(1,2) \mapsto 2$, and $(2,2) \mapsto 3$. The new entries for $F_{\pi}$ are calculated by applying the permutation $\pi$ to the values in $F$ at these locations yielding
\begin{align*}
F_{\pi}[2,1] &= F\left[\eta^{-1}(\pi(\eta(2,1)))\right] = F[\eta^{-1}(2)] = F[1,2] = 2, \\
F_{\pi}[1,2] &= F\left[\eta^{-1}(\pi(\eta(1,2)))\right] = F[\eta^{-1}(1)] = F[2,1] = 3, \\
F_{\pi}[2,2] &= F\left[\eta^{-1}(\pi(\eta(2,2)))\right] = F[\eta^{-1}(3)] = F[2,2] = 1.
\end{align*}
The entry at $(1,3)$ remains unchanged. The resulting filling is
\[
F_{\pi} =
\ytableausetup{notabloids}
\begin{ytableau}
\none & 3 & 1 \\
2 & 1
\end{ytableau}.
\]
\end{example}

\subsection{Universal Skew Schur Modules}

In this section, we develop a concrete construction for the skew Schur module $E^{\lambda/\mu}$, where $E$ is a finitely generated free $R$-module, and $R$ is a commutative ring. Our approach begins by defining the module abstractly through a universal property. This defines
$E^{\lambda/\mu}$ as the universal object for maps originating from $E^{n}$.
This abstract definition is then translated into a concrete algebraic structure.

\begin{definition}\label{def:E_times_lambda_mu}
Let $\lambda/\mu$ be a skew partition and define
\[
E^{\times \lambda/\mu} := \big\{ \mathbf{v} : D(\lambda/\mu) \to E \big\}.
\]
An element $\mathbf{v} \in E^{\times \lambda/\mu}$ assigns a vector $\mathbf{v}(r,c) \in E$ to each box $(r,c) \in D(\lambda/\mu)$.
\end{definition}

Equip $E^{\times \lambda/\mu}$ with the pointwise $R$-module structure: for $\mathbf{v},\mathbf{w}\in E^{\times \lambda/\mu}$ and $r\in R$,
\[
(\mathbf{v}+\mathbf{w})(x)=\mathbf{v}(x)+\mathbf{w}(x),\qquad (r\mathbf{v})(x)=r\,\mathbf{v}(x)\quad\text{for all }x\in D(\lambda/\mu).
\]
Fix a total order $\prec$ on $D(\lambda/\mu)$ and write $D(\lambda/\mu)=\{x_1,\ldots,x_n\}$ with $x_1\prec\cdots\prec x_n$, where $n=|\lambda/\mu|$. The map
\[
E^{\times \lambda/\mu}\longrightarrow E^n,\qquad \mathbf{v}\longmapsto \big(\mathbf{v}(x_1),\ldots,\mathbf{v}(x_n)\big),
\]
is an $R$-linear isomorphism with inverse sending $(v_1,\ldots,v_n)$ to the function $x_i\mapsto v_i$.

Fix $c_1, c_2 \in [\lambda_1]$ such that $c_1 < c_2$ and $a, b \in \mathbb{N}$ such that $a \leq \ell'_{c_1}$ and $b \leq \ell'_{c_2}$ with $\eta$ defined as above. To define Garnir action in this setting, let $\pi\in\mathfrak{S}_{a+b}^{a,b}$ and define $\pi\cdot\mathbf{v}\in E^{\times \lambda/\mu}$ by
\[
(\pi \cdot \mathbf{v})(x,y)=
\begin{cases}
\mathbf{v}\!\big(\eta^{-1}(\pi(\eta(x,y)))\big), & (x,y)\in \Gar_{c_1,c_2}^{a,b},\\[3pt]
\mathbf{v}(x,y), & (x,y)\notin \Gar_{c_1,c_2}^{a,b}.
\end{cases}
\]
We extend this action $R$-linearly in $\mathbf{v}$, so for all $r\in R$ and $\mathbf{v},\mathbf{w}\in E^{\times \lambda/\mu}$,
\[
\pi\cdot(\mathbf{v}+\mathbf{w})=\pi\cdot\mathbf{v}+\pi\cdot\mathbf{w},\qquad
\pi\cdot(r\mathbf{v})=r\,(\pi\cdot\mathbf{v}).
\]
Thus $E^{\times \lambda/\mu}$ is a left $R[\mathfrak{S}_{a+b}^{a,b}]$-module via $\mathbf{v}\mapsto\pi\cdot\mathbf{v}$.

To connect this abstract action to the combinatorial one on fillings, let us consider the special case where $E$ is a free $R$-module with a fixed, ordered basis $\mathcal{B} = \{e_1, \dots, e_m\}$. Any filling $F$ (with entries in $\{1, \dots, m\}$) corresponds to a special vector-filling, which we call a \emph{basis-filling} $\mathbf{v}_F \in E^{\times \lambda/\mu}$, defined by
\[
\mathbf{v}_F(r,c) := e_{F(r,c)} \quad \text{for all } (r,c) \in D(\lambda/\mu).
\]
The above action on basis-fillings is consistent with the action on fillings defined earlier, that is, $\pi \cdot \mathbf{v}_F = \mathbf{v}_{\pi(F^{a, b}_{c_1, c_2})}$ for all basis-fillings $\mathbf{v}_F \in E^{\times \lambda/\mu}$ and $\pi\in\mathfrak{S}_{a+b}^{a,b}$.

\begin{example}
Let $\lambda=(3,2)$ and $\mu=(1)$, so the set of boxes is
\[
D(\lambda/\mu) = \{(1,2), (1,3), (2,1), (2,2)\}.
\]
Let $E$ be a free $R$-module with basis $\{e_1,\dots, e_m\}$. A general element $\mathbf{v} \in E^{\times \lambda/\mu}$ is a vector-filling of the form
\[
{\ytableausetup{mathmode,boxsize=1.5em,centertableaux}
\mathbf{v} =
\ytableausetup{notabloids}
\begin{ytableau}
    \none & v_{1,2} & v_{1,3} \\
    v_{2,1} & v_{2,2}
\end{ytableau}}
\]
where $v_{r,c} \in E$. A filling $F$ and its corresponding basis-filling $\mathbf{v}_F$ are
\[
\ytableausetup{mathmode,boxsize=1.25em,centertableaux}
F =
\ytableausetup{notabloids}
\begin{ytableau}
    \none & 1 & 2 \\
    3 & 2
\end{ytableau}
\quad \text{and} \quad
\mathbf{v}_F =
\ytableausetup{notabloids}
\begin{ytableau}
    \none & e_1 & e_2 \\
    e_3 & e_2
\end{ytableau}.
\]
\end{example}

With this basis-independent action defined, we can now state the universal property of the skew Schur module.

\begin{definition}
A tuple $(a,b,c_1,c_2)$ is \emph{$\lambda/\mu$-admissible} if $a,b\in\mathbb{N}$ and $c_1,c_2\in[\lambda_1]$ with $c_1 < c_2$, and $a\le \ell'_{c_1}$, $b\le \ell'_{c_2}$, and $\lambda'_{c_1}-a<\mu'_{c_2}+b$.
\end{definition}

\begin{definition}[Universal Property of the Skew Schur Module]\label{def:universal-property-skew-schur-module}
Let $E$ be a module over a commutative ring $R$ and let $\lambda/\mu$ be a skew partition. The skew Schur module $E^{\lambda/\mu}$ is an $R$-module that is universal for maps $\varphi: E^{\times \lambda/\mu} \to M$ (for any $R$-module $M$) satisfying

\begin{enumerate}
    \item[\textbf{(1)}] \textbf{R-multilinearity:} The map $\varphi$ is $R$-multilinear in the entries of $\lambda/\mu$.

    \item[\textbf{(2)}] \textbf{Column-Alternating Property:} The map $\varphi$ is alternating in the entries of any given column of the diagram $\lambda/\mu$.

    \item[\textbf{(3)}] \textbf{The Garnir Relations:} For any element $\mathbf{v} \in E^{\times \lambda/\mu}$ and all $\lambda/\mu$-admissible $(a,b,c_1,c_2)$, the map $\varphi$ must satisfy
    \[
    \sum_{\pi\in\mathfrak{S}_{a+b}^{a,b}}\sign{\pi}\varphi(\pi \cdot \mathbf{v})=0.
    \]
\end{enumerate}
\end{definition}

\begin{remark}
Consider the slice category of $R$-modules over $E^{\times \lambda/\mu}$, as defined in~\cite[Ch. II, Sec. 6]{ML98}. Let $\mathcal{C}$ be its full subcategory on those arrows $\varphi:E^{\times \lambda/\mu}\to M$ that are $R$-multilinear, column-alternating, and satisfy the Garnir relations.
By the universal property of the construction, the pair $(E^{\lambda/\mu},\iota)$, where $\iota:E^{\times \lambda/\mu}\to E^{\lambda/\mu}$ is the universal map, is an initial object in $\mathcal{C}$~\cite[Ch. III, Sec. 1]{ML98}.
\end{remark}

\begin{lemma}\label{lem:lemma1}
  Fix a skew partition $\lambda / \mu$. If $E$ is a free $R$-module with a basis \{$e_1, \ldots, e_m\}$, then the skew Schur module $E^{\lambda/\mu}$ is isomorphic to the quotient module $M_R / Q$, where $M_R$ is the free $R$-module with a basis consisting of formal symbols $$\{e_F \mid F \text{ is a filling of } \lambda/\mu \text{ with entries in } \{1, \ldots, m\}\},$$ and $Q$ is the submodule generated by three types of generators:
  \begin{enumerate}
    \item[(i)] $e_F$, where the filling $F$ has two identical entries in the same column, 
    \item[(ii)] $e_F + e_{F'}$, where the filling $F'$ is obtained from $F$ by interchanging two entries within the same column, 
    \item[(iii)] $\sum_{\pi\in \mathfrak{S}_{a+b}^{a,b}} \sign{\pi} e_{F_{\pi}}$, where this sum is defined for any initial filling $F$ and any $\lambda/\mu$-admissible $(a,b,c_1,c_2)$.
  \end{enumerate}
  \end{lemma}

  \begin{proof}
  The goal is to prove that $E^{\lambda/\mu}$, defined abstractly by the universal property in Definition~\ref{def:universal-property-skew-schur-module}, is isomorphic to the concrete module $M_R/Q$. We prove this in three steps, showing that both modules arise from the same sequence of universal constructions. Let $n=|\lambda/\mu|$ and fix the total order on $D(\lambda/\mu)$ used above to identify $E^{\times \lambda/\mu}\cong E^{n}$.

  Let $U_1$ be the universal $R$-module satisfying property (1) (multilinearity) for maps from $E^n$. By definition, $U_1:=E^{\otimes n}$ with universal multilinear map $i_1:E^n\to U_1$.

  Let $M_R$ be the free $R$-module on the set of fillings $F$ of $\lambda/\mu$ with entries in $\{1,\ldots,m\}$. Via the fixed order on $D(\lambda/\mu)$ and the basis $\{e_1,\ldots,e_m\}$ of $E$, the assignment $e_F\longmapsto e_{j_1}\otimes\cdots\otimes e_{j_n}$ (where $(j_1,\ldots,j_n)$ are the entries of $F$ in the fixed total order on $D(\lambda/\mu)$) extends to an $R$-linear isomorphism
  \[
  M_R \xrightarrow{\;\cong\;} U_1.
  \]
  We therefore use $M_R$ as a concrete model for $U_1$ in the subsequent steps.

Let $U_2$ be the universal object satisfying properties (1) and (2), with universal map $i_2:E^n\to U_2$. Since $i_2$ is multilinear, the universal property of $U_1$ yields a unique $R$-linear map $p_1:U_1\to U_2$ such that $i_2=p_1\circ i_1$:
\[
\begin{tikzcd}[column sep=large, row sep=large]
E^n \arrow[r,"i_2"] \arrow[d,"i_1"'] & U_2 \\
U_1 \arrow[ur,dashed,"\exists!\,p_1"'] &
\end{tikzcd}
\]
Observe that $\operatorname{span}(\operatorname{Im}(i_1)) = U_1$ by multilinearity and basis-generation, and by $R$-linearity $p_1(\operatorname{span}(\operatorname{Im}(i_1))) = \operatorname{span}(p_1(\operatorname{Im}(i_1))) = \operatorname{span}(\operatorname{Im}(i_2))$. If $\operatorname{span}(\operatorname{Im}(i_2))$ were a proper submodule of $U_2$, then $i_2$ would factor through this submodule and itself satisfy the same universal property, contradicting the definition of $U_2$. Thus $\operatorname{span}(\operatorname{Im}(i_2)) = U_2$. Combining we have $\operatorname{Im}(p_1) = p_1(U_1) = p_1(\operatorname{span}(\operatorname{Im}(i_1))) = \operatorname{span}(p_1(\operatorname{Im}(i_1))) = \operatorname{span}(\operatorname{Im}(i_2)) = U_2$, that is, $p_1$ is surjective. By the First Isomorphism Theorem,
\[
U_2 \cong U_1/\ker(p_1).
\]
An element $x\in U_1$ lies in $\ker(p_1)$ if and only if it is sent to zero by every multilinear, column-alternating map; thus $\ker(p_1)$ is precisely the submodule generated by the column-alternating relations. In the concrete model, this identifies $U_2$ with $M_R/Q_{\mathrm{alt}}$, where $Q_{\mathrm{alt}}$ is generated by the relations of type (i) and (ii) in Lemma~\ref{lem:lemma1}. Equivalently, $U_2 \cong \bigotimes_{j=1}^{\lambda_1}\!\bigwedge^{\ell'_j}E$ via column-wise exteriorization.

  The skew Schur module $E^{\lambda/\mu}$ is the universal object for maps satisfying all three properties: multilinearity, column-alternating, and the Garnir relations. Let $i_3: E^{n} \to E^{\lambda/\mu}$ be the universal map. By the universal property of $U_2$, there exists a unique $R$-linear map $p_2: U_2 \to E^{\lambda/\mu}$ such that $i_3 = p_2 \circ i_2$. Thus
    \[
    \begin{tikzcd}[column sep=large, row sep=large]
    E^n \arrow[r, "i_3"] \arrow[d, "i_2"'] & E^{\lambda/\mu} \\
    U_2 \arrow[ur, dashed, "\exists!\ p_2"'] &
    \end{tikzcd}.
    \]
    By the same span-and-surjectivity argument as in Step~2, the map $p_2$ is surjective. Hence, by the First Isomorphism Theorem, $E^{\lambda/\mu}\cong U_2/\ker(p_2)$.

An element $y\in U_2$ belongs to $\ker(p_2)$ if and only if $\psi(y)=0$ for every $R$-module $M$ and every $R$-linear map $\psi:U_2\to M$ such that $\varphi=\psi\circ i_2:E^n\to M$ satisfies the Garnir relations for all $\lambda/\mu$-admissible tuples $(a,b,c_1,c_2)$. Thus, for every $\mathbf{v}\in E^n$ and every $\lambda/\mu$-admissible $(a,b,c_1,c_2)$,
\begin{align*}
\sum_{\pi\in \mathfrak{S}_{a+b}^{a,b}}\sign{\pi}\,\varphi(\pi\cdot \mathbf{v})
\;=\;
\psi\!\left(\sum_{\pi\in \mathfrak{S}_{a+b}^{a,b}}\sign{\pi}\, i_2(\pi\cdot \mathbf{v})\right)
\;=\;0.
\end{align*}
Since this holds for every such $\psi$, $\ker(p_2)$ is the submodule of $U_2$ generated by the elements $\sum_{\pi\in \mathfrak{S}_{a+b}^{a,b}}\sign{\pi}\, i_2(\pi\!\cdot\! \mathbf{v})$ as $\mathbf{v}$ ranges over $E^n$ and $(a,\!b,\!c_1,\!c_2)$ ranges over $\lambda/\mu$-admissible tuples.

To simplify this set of generators, for each $\lambda/\mu$-admissible tuple $(a,b,c_1,c_2)$ let
$\delta^{a,b}_{c_1,c_2}: E^n \to U_2$ be the map defined by
$\delta^{a,b}_{c_1,c_2}(\mathbf{v}) = \sum_{\pi} \sign{\pi}\, i_2(\pi \cdot \mathbf{v})$.
The kernel of $p_2$ is the submodule spanned by the images of all such maps $\delta^{a,b}_{c_1,c_2}$.
Each $\delta^{a,b}_{c_1,c_2}$ is $R$-multilinear because $i_2$ is multilinear and the permutation action $\mathbf{v}\mapsto \pi\cdot\mathbf{v}$ is $R$-linear in each component.
Since $E$ is a free module, $E^n$ has a product basis consisting of standard basis $n$-tuples; by $R$-multilinearity, the image of each $\delta^{a,b}_{c_1,c_2}$ is spanned by its values on those basis $n$-tuples~\cite[Corollary~10.16]{03DF}.
Via the fixed order on $D(\lambda/\mu)$, these basis $n$-tuples correspond to the basis-fillings $\mathbf{v}_F$.
Accordingly, $\ker(p_2)$ is generated by the set $\{\delta^{a,b}_{c_1,c_2}(\mathbf{v}_F) \mid \mathbf{v}_F \text{ is a basis-filling and } (a,b,c_1,c_2) \text{ is } \lambda/\mu\text{-admissible}\}$.

    Let $Q_{\mathrm{Garnir}} \subset M_R$ be the submodule generated by \textit{(iii)}. Hence, it is generated by elements of the form $\sum_{\pi} \sign{\pi} e_{F_{\pi}}$. Under the isomorphism $U_2 \cong M_R/Q_{\mathrm{alt}}$, the generators $\delta^{a,b}_{c_1,c_2}(\mathbf{v}_F)=\sum_{\pi}\sign{\pi}\, i_2(\mathbf{v}_{F_\pi})$ of $\ker(p_2)$, over all $\lambda/\mu$-admissible $(a,b,c_1,c_2)$, correspond precisely to the cosets of the generators of $Q_{\mathrm{Garnir}}$. Therefore, the submodule $\ker(p_2)$ corresponds to the submodule $(Q_{alt} + Q_{\mathrm{Garnir}}) / Q_{alt}$.

    Finally, by the Third Isomorphism Theorem for modules, we have
    \[
    E^{\lambda/\mu} \cong U_2 / \ker(p_2) \cong (M_R / Q_{alt}) / \left( (Q_{alt} + Q_{\mathrm{Garnir}}) / Q_{alt} \right) \cong M_R / (Q_{alt} + Q_{\mathrm{Garnir}}).
    \]
    The submodule $Q_{alt} + Q_{\mathrm{Garnir}}$ is, by definition, equal to $Q$. We conclude that $E^{\lambda/\mu}\!\cong\!M_R / Q$.
  \end{proof}

  \begin{corollary}\label{cor:explicit_construction}
Fix a skew partition $\lambda/\mu$. The skew Schur module $E^{\lambda/\mu}$ is isomorphic to the quotient module:
\[
E^{\lambda/\mu} \cong \frac{\bigotimes_{j=1}^{\lambda_1} \bigwedge^{\ell'_j}(E)}{Q(\lambda/\mu, E)},
\]
where $Q(\lambda/\mu, E)$ is the submodule generated by the Garnir relations.
\end{corollary}

\begin{proof}
This follows directly from the argument presented in steps 2 and 3 of the proof of Lemma~\ref{lem:lemma1}. Step 2 identifies the universal object for properties (1) and (2) with the tensor product of exterior powers, $\bigotimes_{j=1}^{\lambda_1} \bigwedge^{\ell'_j}(E)$. Step 3 then shows that the skew Schur module is the quotient of this very object by the submodule generated by the Garnir relations.
\end{proof}

\subsection{A Determinantal identity} Throughout this section, fix a skew parition $\lambda/\mu$, and  fix an ordered basis $(e_1,\ldots,e_m)$ of the free $R$-module $E$, so all filling entries lie in $[m]=\{1,\ldots,m\}$. Let $Z_{i,j}$ for $1 \leq i \leq \lambda'_1$ and $1 \leq j \leq m$ be a set of indeterminates. Our next objective is to construct a concrete basis for the quotient module $M_R/Q$ by exhibiting an $R$-module homomorphism $E^{\lambda/\mu}\to R[Z_{i,j}]$. The present subsection develops the technical matrix identities needed to define this map and verify that it respects the relations (i)-(iii) of Lemma~\ref{lem:lemma1}. This mirrors the partition case treated by Fulton~\cite{97FW}, where an analogous map is validated using a classical determinantal identity of Sylvester. In the skew setting, the required compatibility does not follow from Sylvester's identity, instead we establish a determinantal identity tailored to our construction.

Throughout this paper, we frequently work with matrices whose subscripts and superscripts involve multiple parameters. To minimize notational clutter when referring to their entries, we adopt the following convention: for a matrix $A$ and indices $i, j$, the entry in the $i$-th row and $j$-th column is denoted by $A[i,j]$ rather than the conventional $A_{i,j}$.
Fix a skew partition $\lambda / \mu$. We define a family of sequences corresponding to the row indices found in fixed columns of $\lambda / \mu$, namely for $c \in [\lambda_1]$ let
\begin{align*}
\row_c & :=  (\row_{c,1},\ldots,\row_{c,\ell'_c}) := (\mu'_c + 1, \ldots, \lambda'_c).
\end{align*}
Thus, $\row_c$ is the sequence of row indices of all boxes in column $c$, and $\row_{c,i}$ is the row index of the $i$th box (from the top) in column $c$ of $\lambda / \mu$.

Let $F$ be a fixed filling of shape $\lambda / \mu$.
For each $c\in [\lambda_1]$, we define a $ \ell'_c \times \ell'_c $ matrix $ M_{F, c} $ with
  \begin{equation}\label{def:M_F_c}
  M_{F, c}[i,j] = Z_{\row_{c,i},F[\row_{c,j}, c]}
  \end{equation}
for $1 \leq i,j \leq \ell'_c$. Thus, the (i,j) entry of $M_{F, c}$ is exactly the indeterminate indexed by the row of the $i$th box (from the top) in column $c$ of $\lambda / \mu$ and the value of $F$ in the row of the $j$th box (from the top) in column $c$.

\begin{example}
\label{ex:fillingColtoMat}
Let $F$ be a filling of the shape $(3,3,3,3,3,2,2,1)/(2, 1,1)$,
\[
F = \begin{ytableau}
\none & \none & 8 \\
\none & 10 & 5 \\
\none & 11 & 6 \\
14    & 2   & 7 \\
18    & 1   & 9 \\
16    & 3   & \none \\
17    & 4   & \none \\
15    & \none & \none
\end{ytableau}.
\]
Let $c_1 = 2$ and $c_2 = 3$. Then $\ell'_{c_1} = 6, \ell'_{c_2} = 5$, and
\[
 M_{F, 2} = \begin{bmatrix}
  Z_{2, 10} & Z_{2, 11} & Z_{2, 2} & Z_{2, 1} & Z_{2, 3} & Z_{2, 4} \\
  Z_{3, 10} & Z_{3, 11} & Z_{3, 2} & Z_{3, 1} & Z_{3, 3} & Z_{3, 4} \\
  Z_{4, 10} & Z_{4, 11} & Z_{4, 2} & Z_{4, 1} & Z_{4, 3} & Z_{4, 4} \\
  Z_{5, 10} & Z_{5, 11} & Z_{5, 2} & Z_{5, 1} & Z_{5, 3} & Z_{5, 4} \\
  Z_{6, 10} & Z_{6, 11} & Z_{6, 2} & Z_{6, 1} & Z_{6, 3} & Z_{6, 4} \\
  Z_{7, 10} & Z_{7, 11} & Z_{7, 2} & Z_{7, 1} & Z_{7, 3} & Z_{7, 4} \\
 \end{bmatrix},
\quad
 M_{F, 3} = \begin{bmatrix}
  Z_{1, 8} & Z_{1, 5} & Z_{1, 6} & Z_{1, 7} & Z_{1, 9} \\
  Z_{2, 8} & Z_{2, 5} & Z_{2, 6} & Z_{2, 7} & Z_{2, 9} \\
  Z_{3, 8} & Z_{3, 5} & Z_{3, 6} & Z_{3, 7} & Z_{3, 9} \\
  Z_{4, 8} & Z_{4, 5} & Z_{4, 6} & Z_{4, 7} & Z_{4, 9} \\
  Z_{5, 8} & Z_{5, 5} & Z_{5, 6} & Z_{5, 7} & Z_{5, 9} \\
 \end{bmatrix}.
\]
\end{example}

Finally, we define $D_{F, c}$ to be the determinant of $M_{F, c}$.
For a filling $F$ with shape $\lambda / \mu$, we let $D_F$ be the product of the determinants corresponding to the columns of $F$, namely
  \begin{equation}\label{def:D_F}
  D_F=\prod_{c=1}^{\lambda_1} D_{F, c}.
  \end{equation}

To model a specific Garnir relation involving columns $c_1, c_2$ and parameters $a, b$, we construct a composite matrix. We define two off-diagonal block matrices, $\overline{M}_{F,c_1}^{a}$ and $\overline{M}_{F,c_2}^{b}$, whose entries depend on the Garnir sets $\Gar^{a,\_}_{c_1,\_}$ and $\Gar^{\_,b}_{\_,c_2}$.

The matrix $\overline{M}_{F,c_1}^{a}$ is an $\ell'_{c_2} \times \ell'_{c_1}$ matrix defined by
\[
\overline{M}_{F,c_1}^{a}[i,j] := \begin{cases}
  Z_{\row_{c_2,i},F[\row_{c_1,j}, c_1]} & \text{if } (\row_{c_1,j}, c_1) \in \Gar^{a,\_}_{c_1,\_} \\
  0 & \text{otherwise}
\end{cases}.
\]
The matrix $\overline{M}_{F,c_2}^{b}$ is an $\ell'_{c_1} \times \ell'_{c_2}$ matrix defined by
\[
\overline{M}_{F,c_2}^{b}[i,j] := \begin{cases}
  Z_{\row_{c_1,i},F[\row_{c_2,j}, c_2]} & \text{if } (\row_{c_2,j}, c_2) \in \Gar^{\_,b}_{\_,c_2} \\
  0 & \text{otherwise}
\end{cases}.
\]

The full matrix $K_{F}^{c_1, c_2, a, b}$ is the $(\ell'_{c_1}+\ell'_{c_2}) \times (\ell'_{c_1}+\ell'_{c_2})$ block matrix given by
  \[
  K_{F}^{c_1, c_2, a, b} := \begin{bmatrix}
    M_{F, c_1} & \overline{M}_{F, c_2}^{b} \\
    \overline{M}_{F, c_1}^{a} & M_{F, c_2}
    \end{bmatrix}.
  \]

\begin{example} We continue with the filling $F$ from Example~\ref{ex:fillingColtoMat}. Let $c_1 = 2$, $c_2 = 3$, $a=4$, and $b=5$.
The Garnir set for column $c_1=2$ involves its bottom $a=4$ cells, which are those in rows $\{4,5,6,7\}$. The Garnir set for column $c_2=3$ involves its top $b=5$ cells, which are those in rows $\{1,2,3,4,5\}$.
This gives the off-diagonal blocks
  \[
  \overline{M}_{F, 2}^{4} = \begin{bmatrix}
    0 & 0 & Z_{1, 2} & Z_{1, 1} & Z_{1, 3} & Z_{1, 4} \\
    0 & 0 & Z_{2, 2} & Z_{2, 1} & Z_{2, 3} & Z_{2, 4} \\
    0 & 0 & Z_{3, 2} & Z_{3, 1} & Z_{3, 3} & Z_{3, 4} \\
    0 & 0 & Z_{4, 2} & Z_{4, 1} & Z_{4, 3} & Z_{4, 4} \\
    0 & 0 & Z_{5, 2} & Z_{5, 1} & Z_{5, 3} & Z_{5, 4} \\
   \end{bmatrix},
  \quad
  \overline{M}_{F, 3}^{5} = \begin{bmatrix}
    Z_{2, 8} & Z_{2, 5} & Z_{2, 6} & Z_{2, 7} & Z_{2, 9} \\
    Z_{3, 8} & Z_{3, 5} & Z_{3, 6} & Z_{3, 7} & Z_{3, 9} \\
    Z_{4, 8} & Z_{4, 5} & Z_{4, 6} & Z_{4, 7} & Z_{4, 9} \\
    Z_{5, 8} & Z_{5, 5} & Z_{5, 6} & Z_{5, 7} & Z_{5, 9} \\
    Z_{6, 8} & Z_{6, 5} & Z_{6, 6} & Z_{6, 7} & Z_{6, 9} \\
    Z_{7, 8} & Z_{7, 5} & Z_{7, 6} & Z_{7, 7} & Z_{7, 9} \\
   \end{bmatrix}.
  \]
  Then $K_{F}^{2, 3, 4, 5}$ is the $11 \times 11$ matrix constructed by assembling $M_{F,2}$, $M_{F,3}$, $\overline{M}_{F,3}^{5}$, and $\overline{M}_{F,2}^{4}$ into a block matrix as defined above.

\end{example}

The next few pages assemble a chain of technical lemmas whose purpose is to relate our determinantal data $D_{F,c}$ (and $D_F=\prod_c D_{F,c}$) to the Garnir action on fillings. We analyze the block matrix $K_F^{c_1,c_2,a,b}$ and its zeroes to show that, under the Garnir overlap condition, a signed sum of the two-column determinants $D_{F_\pi,c_1}D_{F_\pi,c_2}$ vanishes, and consequently we are able to show in Corollary~\ref{cor:garnir_determinant_relation} that the global signed sum $\sum_{\pi \in \mathfrak{S}_{a+b}^{a,b}} \sign{\pi} \,D_{F_\pi}$ also vanishes. This establishes a precise determinantal avatar of the Garnir relations. The resulting identity is the key input for the next subsection, where it enables the construction of an $R$-module homomorphism from the skew Schur module to a polynomial ring.

\begin{lemma}\label{lemma:overlap}
  Fix a skew partition $\lambda/\mu$. Choose $c_1, c_2 \in [\lambda_1]$ such that $c_1 < c_2$ and $a, b \in \mathbb{N}$ such that $a \leq \ell'_{c_1}$ and $b \leq \ell'_{c_2}$. Let $\mathrm{set}(\row_c)$ denote the set of unique entries in the sequence $\row_c$. If $\lambda'_{c_1} - a < \mu'_{c_2} + b$, then
  \[
    |\mathrm{set}(\row_{c_1}) \cap \mathrm{set}(\row_{c_2})| > (\ell'_{c_1} - a) + (\ell'_{c_2} - b).
  \]
\end{lemma}

\begin{proof}
Since $c_1 < c_2$, $\lambda'_{c_1} \ge \lambda'_{c_2}$ and $\mu'_{c_1} \ge \mu'_{c_2}$. The sets of row indices for these columns are the integer intervals $[\mu'_{c_1} + 1, \lambda'_{c_1}]$ and $[\mu'_{c_2} + 1, \lambda'_{c_2}]$. The intersection of these two intervals is therefore $[\mu'_{c_1} + 1, \lambda'_{c_2}]$, so its size is
\begin{equation}\label{eq:intersection_size}
|\mathrm{set}(\row_{c_1}) \cap \mathrm{set}(\row_{c_2})| = \max(0, \lambda'_{c_2} - \mu'_{c_1}).
\end{equation}
Our goal is to show that this quantity is strictly greater than $(\ell'_{c_1} - a) + (\ell'_{c_2} - b)$. We have
\begin{align*}
\lambda'_{c_2} - \mu'_{c_1} &= (\lambda'_{c_2} - \mu'_{c_1}) + ( (\lambda'_{c_1} - \mu'_{c_1} - a) + (\lambda'_{c_2} - \mu'_{c_2} - b) ) - ( (\ell'_{c_1} - a) + (\ell'_{c_2} - b) ) \nonumber \\
&= (a + b + \mu'_{c_2} - \lambda'_{c_1}) + ((\ell'_{c_1} - a) + (\ell'_{c_2} - b)),
\end{align*}
where in the first equality we have simply added $0$. The first term on the right hand side, $(a + b + \mu'_{c_2} - \lambda'_{c_1})$, is a positive integer due to $\lambda_{c_1}'-a < \mu_{c_2}'+b$. The second term on the right hand side, $((\ell'_{c_1} - a) + (\ell'_{c_2} - b))$, is a non-negative integer because $a \le \ell'_{c_1}$ and $b \le \ell'_{c_2}$. Therefore, we have shown that
\[
\lambda'_{c_2} - \mu'_{c_1} > (\ell'_{c_1} - a) + (\ell'_{c_2} - b) \ge 0.
\]
This inequality proves the lemma, since it also demonstrates that $\lambda'_{c_2} - \mu'_{c_1}$ is positive, and hence \eqref{eq:intersection_size} implies $|\mathrm{set}(\row_{c_1}) \cap \mathrm{set}(\row_{c_2})| = \lambda'_{c_2} - \mu'_{c_1}$.
\end{proof}

\begin{lemma}\label{lemma:singularity}
Fix a filling $F$ of shape $\lambda / \mu$. Choose $c_1, c_2 \in [\lambda_1]$ such that $c_1 < c_2$ and $a, b \in \mathbb{N}$ such that $a \leq \ell'_{c_1}$ and $b \leq \ell'_{c_2}$.
If $\lambda_{c_1}'-a < \mu_{c_2}'+b$ holds, then the matrix $K_{F}^{c_1, c_2, a, b}$ is singular.
\end{lemma}
\begin{proof}
Let $K=K_{F}^{c_1, c_2, a, b}$. By Lemma~\ref{lemma:overlap}, $\lambda_{c_1}'-a < \mu_{c_2}'+b$ ensures that the set of overlapping diagram row indices, $I_{\text{int}} = \mathrm{set}(\row_{c_1}) \cap \mathrm{set}(\row_{c_2})$, is non-empty. Let $M = |I_{\text{int}}| = \lambda'_{c_2} - \mu'_{c_1}$.

For each diagram row index $r \in I_{\text{int}}$, we identify two specific rows in the matrix $K$,
\begin{itemize}
    \item[(i)] let $R_r^{\text{top}}$ be the row, represented as a $1\times (\ell'_{c_1}+ \ell'_{c_2})$ matrix,  in the top half of $K$ corresponding to the diagram row $r$ from column $c_1$. This is the row with index $r - \mu'_{c_1}$ in $K$,
    \item[(ii)] let $R_r^{\text{bot}}$ be the row, represented as a $1\times (\ell'_{c_1}+ \ell'_{c_2})$ matrix, in the bottom half of $K$ corresponding to the diagram row $r$ from column $c_2$. This is the row with index $\ell'_{c_1} + (r - \mu'_{c_2})$ in $K$.
\end{itemize}
The entries of these two rows are explicitly given by the block-matrix definition of $K$. Let $1 \le j \le \ell'_{c_1}$ and $1 \le k \le \ell'_{c_2}$, then
\begin{align*}
R_r^{\text{top}}[1,j] &= M_{F,c_1}[r - \mu'_{c_1}, j] = Z_{r, F[\row_{c_1,j}, c_1]} \\
R_r^{\text{bot}}[1,j] &= \overline{M}_{F,c_1}^a[r - \mu'_{c_2}, j] = \begin{cases} Z_{r, F[\row_{c_1,j}, c_1]} & \text{if } (\row_{c_1,j},c_1) \in \Gar^{a,\_}_{c_1,\_} \\ 0 & \text{otherwise} \end{cases} \\
R_r^{\text{top}}[1,\ell'_{c_1}+k] &= \overline{M}_{F,c_2}^b[r - \mu'_{c_1}, k] = \begin{cases} Z_{r, F[\row_{c_2,k}, c_2]} & \text{if } (\row_{c_2,k},c_2) \in \Gar^{\_,b}_{\_, c_2} \\ 0 & \text{otherwise} \end{cases} \\
R_r^{\text{bot}}[1,\ell'_{c_1}+k] &= M_{F,c_2}[r - \mu'_{c_2}, k] = Z_{r, F[\row_{c_2,k}, c_2]}
\end{align*}
We perform $M$ row operations on $K$ by replacing each row $R_r^{\text{top}}$ with the new vector $V_r = R_r^{\text{top}} - R_r^{\text{bot}}$. Let the resulting matrix be $K'$. Performing these elementary row operations, replacing each $R^{\mathrm{top}}_r$ by $R^{\mathrm{top}}_r-R^{\mathrm{bot}}_r$, does not change the determinant, so $\det(K')=\det(K)$. We will show that the set of $M$ new row vectors $\{V_r\}_{r \in I_{\text{int}}}$ is linearly dependent.

Examining the vector $V_r$, its entries are zero for any column corresponding to a cell in the Garnir set $\Gar_{c_1,c_2}^{a,b} = \Gar^{a,\_}_{c_1,\_} \cup \Gar^{\_,b}_{\_,c_2}$ since
\begin{itemize}
    \item[(i)] for a column $j$ corresponding to a cell in $\Gar^{a,\_}_{c_1,\_}$, we have $R_r^{\text{top}}[1,j] = Z_{r, F[\row_{c_1,j}, c_1]}$ and $R_r^{\text{bot}}[1,j] = Z_{r, F[\row_{c_1,j}, c_1]}$;
    \item[(ii)] for a column $\ell'_{c_1}+k$ corresponding to a cell in $\Gar^{\_,b}_{\_,c_2}$, we have $R_r^{\text{top}}[1,\ell'_{c_1}+k] = Z_{r, F[\row_{c_2,k}, c_2]}$ and $R_r^{\text{bot}}[1,\ell'_{c_1}+k] = Z_{r, F[\row_{c_2,k}, c_2]}$.
\end{itemize}
In both cases, the entries from $R_r^{\text{top}}$ and $R_r^{\text{bot}}$ are identical for columns associated with the Garnir set, and thus they cancel in $V_r$.

That is, each $V_r$ vanishes on the columns that do not correspond to cells in the Garnir set, so all $\{V_r\}_{r\in I_{\mathrm{int}}}$ lie in the $R$-span of the canonical basis vectors corresponding to the $d$ non-columns that correspond to cells in the Garnir set. This submodule is free of rank $d$. Now
\[
d = (\ell'_{c_1} - a) + (\ell'_{c_2} - b),
\]
and so by Lemma~\ref{lemma:overlap}, $\lambda_{c_1}'-a < \mu_{c_2}'+b$ implies $M > d$.  Any $M$ vectors in a free $R$-module of rank $d$ are linearly dependent. Hence $\{V_r\}_{r \in I_{\text{int}}}$ is linearly dependent, so $\det(K') = 0$ and $K$ is singular.
\end{proof}

\begin{definition}\label{def:embeddings}
For any composition $c = (k_1, \dots, k_p)$ of an integer $n$, let the corresponding Young subgroup be denoted by $\mathfrak{S}_c := \mathfrak{S}_{k_1} \times \cdots \times \mathfrak{S}_{k_p}$. Let
\[
\iota_c: \mathfrak{S}_c \hookrightarrow \mathfrak{S}_n
\]
denote the standard embedding of this subgroup into $\mathfrak{S}_n$.

The proofs in this section will make repeated use of two instances of this embedding.
\begin{enumerate}
    \item The map $\varphi$ is the standard embedding corresponding to the composition $(\ell'_{c_1}, \ell'_{c_2})$:
    \[
    \varphi := \iota_{(\ell'_{c_1}, \ell'_{c_2})}: \mathfrak{S}_{\ell'_{c_1}}\times \mathfrak{S}_{\ell'_{c_2}}\to \mathfrak{S}_{\ell'_{c_1}+\ell'_{c_2}}.
    \]

    \item Let $a \le \ell'_{c_1}$ and $b \le \ell'_{c_2}$. The map $\psi: \mathfrak{S}_{a+b}\to \mathfrak{S}_{\ell'_{c_1}+\ell'_{c_2}}$ embeds a permutation into the middle block corresponding to the composition $c' = (\ell'_{c_1}-a, a+b, \ell'_{c_2}-b)$. It is defined as
    \[
    \psi(\pi) := \iota_{c'}(\operatorname{id}_{\ell'_{c_1}-a}, \pi, \operatorname{id}_{\ell'_{c_2}-b}).
    \]
\end{enumerate}
Both $\varphi$ and $\psi$ are injective homomorphisms that preserve the sign of the input permutations. The image of the set of shuffles $\mathfrak{S}_{a+b}^{a, b}$ under the map $\psi$ will be denoted $\overline{\mathfrak{S}}_{a+b}^{a, b}$.
\end{definition}

\begin{example}
Suppose that $\ell'_{c_1} = 3$ and $\ell'_{c_2} = 3$, and $a = 2$ and $b = 3$.
Consider $(1, 2, 4, 3, 5) \in \mathfrak{S}_{2+3}^{2, 3}$. Then the embedding under $\psi$ is
\[
\psi((1, 2, 4, 3, 5)) = (1, 2, 3, 5, 4, 6)\in \overline{\mathfrak{S}_{2+3}^{2, 3}}.
\]
\end{example}

\begin{lemma}\label{lemma:bijection}
Fix a skew partition $\lambda/\mu$. Choose columns $c_1, c_2 \in [\lambda_1]$ with $c_1 < c_2$, and integers $a, b \in \mathbb{N}$ such that $a \leq \ell'_{c_1}$ and $b \leq \ell'_{c_2}$. Let $S = \overline{\mathfrak{S}_{a+b}^{a, b}} \cdot \varphi(\mathfrak{S}_{\ell'_{c_1}}\times \mathfrak{S}_{\ell'_{c_2}})$ denote the setwise product of these elements in $\mathfrak{S}_{\ell'_{c_1}+\ell'_{c_2}}$. The map
\[
\mathcal{F}: \overline{\mathfrak{S}_{a+b}^{a, b}} \times \varphi(\mathfrak{S}_{\ell'_{c_1}}\times \mathfrak{S}_{\ell'_{c_2}}) \to S,
\]
defined by $\mathcal{F}(c, h) = c \cdot h$, is a bijection. In other words, the set $\overline{\mathfrak{S}}_{a+b}^{a, b}$ is a left transversal for the subgroup $\varphi(\mathfrak{S}_{\ell'_{c_1}}\times \mathfrak{S}_{\ell'_{c_2}})$ in the set $S$.
\end{lemma}
\begin{proof}
The map $\mathcal{F}$ is surjective by definition. To prove it is a bijection, we must show that if $c_1, c_2 \in \overline{\mathfrak{S}_{a+b}^{a, b}}$ and $h_1, h_2 \in \varphi(\mathfrak{S}_{\ell'_{c_1}}\times \mathfrak{S}_{\ell'_{c_2}})$ satisfy $c_1 h_1 = c_2 h_2$, then $c_1 = c_2$ and $h_1 = h_2$.

The condition $c_1 h_1 = c_2 h_2$ is equivalent to $c_2^{-1} c_1 = h_2 h_1^{-1}$. This implies, since $\varphi(\mathfrak{S}_{\ell'_{c_1}}\times \mathfrak{S}_{\ell'_{c_2}})$ is a subgroup, that $h_2 h_1^{-1} \in \varphi(\mathfrak{S}_{\ell'_{c_1}}\times \mathfrak{S}_{\ell'_{c_2}})$ and consequently $c_2^{-1} c_1 \in \varphi(\mathfrak{S}_{\ell'_{c_1}}\times \mathfrak{S}_{\ell'_{c_2}})$.

Let $c_1 = \psi(\pi_1)$ and $c_2 = \psi(\pi_2)$, where $\pi_1, \pi_2 \in \mathfrak{S}_{a+b}^{a,b}$ are $(a,b)$-shuffles. Since $\psi$ is a homomorphism, $c_2^{-1} c_1 = \psi(\pi_2^{-1}\pi_1)$. Thus,
\begin{equation}\label{eq:inimage}
\psi(\pi_2^{-1}\pi_1) \in \varphi(\mathfrak{S}_{\ell'_{c_1}}\times \mathfrak{S}_{\ell'_{c_2}}).
\end{equation}

Let $g \in \mathfrak{S}_{a+b}$. We now show that $\psi(g) \in \varphi(\mathfrak{S}_{\ell'_{c_1}}\times \mathfrak{S}_{\ell'_{c_2}})$ if and only if $g \in \mathfrak{S}_a \times \mathfrak{S}_b$. The condition $\psi(g)\in \varphi(\mathfrak{S}_{\ell'_{c_1}}\times \mathfrak{S}_{\ell'_{c_2}})$ holds if and only if $\psi(g)$ preserves each block $L_1=\{1,\ldots,\ell'_{c_1}\}$ and $L_2=\{\ell'_{c_1}+1,\ldots,\ell'_{c_1}+\ell'_{c_2}\}$. Since $\psi$ fixes the first $\ell'_{c_1}-a$ positions and the last $\ell'_{c_2}-b$ positions, this is equivalent to requiring that $\psi(g)$ preserve
\[
A=\{\ell'_{c_1}-a+1,\ldots,\ell'_{c_1}\}\subset L_1\quad\text{and}\quad
B=\{\ell'_{c_1}+1,\ldots,\ell'_{c_1}+b\}\subset L_2.
\]
By definition of $\psi$, for $1\le j\le a$ one has $\psi(g)((\ell'_{c_1}-a)+j)=(\ell'_{c_1}-a)+g(j)$, and for $1\le j\le b$ one has $\psi(g)(\ell'_{c_1}+j)=\ell'_{c_1}+g(a+j)$. Hence $\psi(g)$ preserves $A$ and $B$ if and only if $g$ preserves $\{1,\ldots,a\}$ and $\{a+1,\ldots,a+b\}$, that is, $g\in \mathfrak{S}_a\times \mathfrak{S}_b$.

The above equivalence and \eqref{eq:inimage} imply that $\pi_2^{-1}\pi_1 \in \mathfrak{S}_a \times \mathfrak{S}_b$. Thus, $\pi_1$ and $\pi_2$ belong to the same left coset of $\mathfrak{S}_a \times \mathfrak{S}_b$ in $\mathfrak{S}_{a+b}$. It is a standard result in the theory of symmetric groups that $\mathfrak{S}_{a+b}^{a,b}$ forms a complete and unique set of representatives for the left cosets of $\mathfrak{S}_a \times \mathfrak{S}_b$ in $\mathfrak{S}_{a+b}$ \cite[Proposition 2.4.4, Corollary 2.4.5]{BB05}. Thus, if $\pi_1$ and $\pi_2$ belong to the same coset, they must be identical.

Since $\psi$ is an injective map, $\pi_1 = \pi_2$ implies $c_1 = \psi(\pi_1) = \psi(\pi_2) = c_2$. Substituting this back into the original equation gives $c_1 h_1 = c_1 h_2$, which implies $h_1 = h_2$.
\end{proof}

\begin{lemma}\label{lemma:zero}
Fix a filling $F$ of shape $\lambda / \mu$.  Choose columns $c_1, c_2 \in [\lambda_1]$ with $c_1 < c_2$, and integers $a, b \in \mathbb{N}$ such that $a \leq \ell'_{c_1}$ and $b \leq \ell'_{c_2}$. Set $S =  \overline{\mathfrak{S}_{a+b}^{a, b}}\cdot \varphi\!\left(\mathfrak{S}_{\ell'_{c_1}}\times \mathfrak{S}_{\ell'_{c_2}}\right)\;\subseteq\;\mathfrak{S}_{\ell'_{c_1}+\ell'_{c_2}}$. Let $A=\{\sigma\in \mathfrak{S}_{\ell'_{c_1}+\ell'_{c_2}} \mid K_{F}^{c_1,c_2,a,b}[i,\sigma(i)]\neq 0 \text{ for all } i\in[\ell'_{c_1}+\ell'_{c_2}]\}$. Then $\sigma\in A$ if and only if $\sigma\in S$.
\end{lemma}

\begin{proof}
Write $K=K_{F}^{c_1,c_2,a,b}$ and define the row index sets $T=\{1,\dots,\ell'_{c_1}\}$ and $B=\{\ell'_{c_1}+1,\dots,\ell'_{c_1}+\ell'_{c_2}\}$, and the column index sets
\[
L=\{1,\dots,\ell'_{c_1}-a\},\quad M=\{\ell'_{c_1}-a+1,\dots,\ell'_{c_1}+b\},\quad R=\{\ell'_{c_1}+b+1,\dots,\ell'_{c_1}+\ell'_{c_2}\}.
\]
With respect to these partitions, $K$ has the $2\times 3$ block form
\[
K \;=\;
\left[
\begin{array}{c|c|c}
K_{TL} & K_{TM} & 0_{\ell'_{c_1}\times(\ell'_{c_2}-b)}\\\hline
0_{\ell'_{c_2}\times(\ell'_{c_1}-a)} & K_{BM} & K_{BR}
\end{array}
\right],
\]
where $K_{TL}\in \mathrm{Mat}_{\ell'_{c_1}\times(\ell'_{c_1}-a)}$, $K_{TM}\in \mathrm{Mat}_{\ell'_{c_1}\times(a+b)}$, $K_{BM}\in \mathrm{Mat}_{\ell'_{c_2}\times(a+b)}$, and $K_{BR}\in \mathrm{Mat}_{\ell'_{c_2}\times(\ell'_{c_2}-b)}$.

To prove $S\subseteq A$, take $\sigma=c\cdot h$ with $c=\psi(\pi)\in  \overline{\mathfrak{S}_{a+b}^{a, b}}$ and $h=\varphi(\sigma_1,\sigma_2)\in \varphi(\mathfrak{S}_{\ell'_{c_1}}\times \mathfrak{S}_{\ell'_{c_2}})$. Every element of $\mathrm{im}(\psi)$ fixes the columns $L$ and $R$ pointwise and permutes the columns $M$, while every element of $\mathrm{im}(\varphi)$ preserves the row sets $T$ and $B$ setwise. Consequently, if $i\in T$ then $\sigma(i)\in T \subset L\cup M$ and never in $R$, hence $K[i,\sigma(i)]\neq 0$ because the only zeros in the $T$-rows lie in the $T\times R$ block. If $i\in B$ then $\sigma(i)\in B \subset M\cup R$ and never in $L$, hence $K[i,\sigma(i)]\neq 0$ because the only zeros in the $B$-rows lie in the $B\times L$ block. Thus $\sigma\in A$.

To compare cardinalities, $A$ counts the ways to choose coordinates $\{(i,\sigma(i))\}$ with exactly one in each row and each column and all chosen entries nonzero. First choose entries in the $|L|=\ell'_{c_1}-a$ columns of $L$. These must pair with distinct rows in $T$, which is equivalent to choosing the $a$ rows of $T$ that remain unused at this stage, giving $\binom{\ell'_{c_1}}{a}\cdot(\ell'_{c_1}-a)! = \ell'_{c_1}!/a!$ possibilities. These choices only consume rows of $T$ so they do not constrain the next choices in $R$. Next choose entries in the $|R|=\ell'_{c_2}-b$ columns of $R$. These must pair with distinct rows in $B$, which is equivalent to choosing the $b$ rows of $B$ that remain unused at this stage, giving $\binom{\ell'_{c_2}}{b}\cdot(\ell'_{c_2}-b)! = \ell'_{c_2}!/b!$ possibilities. After these two stages, there remain exactly $a$ unused rows in $T$ and $b$ unused rows in $B$, and the unused columns are exactly the $a+b$ columns of $M$, where every entry is structurally nonzero in both $T$ and $B$, so any bijection between the remaining rows and $M$ is valid, contributing $(a+b)!$. Therefore
\[
|A| \;=\; \frac{\ell'_{c_1}!}{a!}\cdot \frac{\ell'_{c_2}!}{b!}\cdot (a+b)! \;=\; \ell'_{c_1}!\,\ell'_{c_2}!\,\binom{a+b}{a}.
\]
By Lemma~\ref{lemma:bijection}, the multiplication map $ \overline{\mathfrak{S}_{a+b}^{a, b}}\times \varphi(\mathfrak{S}_{\ell'_{c_1}}\times \mathfrak{S}_{\ell'_{c_2}})\to S$ is a bijection, whence
\[
|S| \;=\; \big| \overline{\mathfrak{S}_{a+b}^{a, b}}\big| \cdot \big|\varphi(\mathfrak{S}_{\ell'_{c_1}}\times \mathfrak{S}_{\ell'_{c_2}})\big|
\;=\; \binom{a+b}{a}\cdot \ell'_{c_1}!\,\ell'_{c_2}!
\;=\; |A|.
\]
Since $S\subseteq A$ and $|S|=|A|$, we conclude $A=S$.
\end{proof}

\begin{lemma}\label{lemma:matrix_iden}
  Fix a filling $F$ of shape $\lambda / \mu$. Choose $c_1, c_2 \in [\lambda_1]$ such that $c_1 < c_2$ and $a, b \in \mathbb{N}$ such that $a \leq \ell'_{c_1}$ and $b \leq \ell'_{c_2}$. For any $\pi\in \mathfrak{S}_{a+b}^{a, b}$, we have the relationship
\[
K_{F_{\pi}}^{c_1, c_2,a,b}[p, q] = K_{F}^{c_1, c_2,a,b}[p, \psi(\pi)(q)].
\]
\end{lemma}

\begin{proof}
Define the map $\zeta:[\ell'_{c_1}+\ell'_{c_2}]\to D(\lambda/\mu)$ such that
\[
\zeta(j):=
\begin{cases}
(\row_{c_1,j},c_1) & \text{if } 1\le j\le \ell'_{c_1},\\
(\row_{c_2,j-\ell'_{c_1}},c_2) & \text{if } \ell'_{c_1}+1\le j\le \ell'_{c_1}+\ell'_{c_2}.
\end{cases}
\]
We claim that
\begin{equation}\label{eq:filling-transport}
F_{\pi}[\zeta(q)] \;=\; F[\zeta(\psi(\pi)(q))]\qquad\text{for all }q.
\end{equation}

Assuming \eqref{eq:filling-transport}, we now deduce the matrix equality
\[
K_{F_{\pi}}^{c_1,c_2,a,b}[p,q]\;=\;K_{F}^{c_1,c_2,a,b}[p,\psi(\pi)(q)]\quad\text{for all }p,q.
\]
The set of zero positions in $K_{F}^{c_1,c_2,a,b}$ is determined solely by $\lambda/\mu$ and $(c_1,c_2,a,b)$ and hence is identical for $K_{F_{\pi}}^{c_1,c_2,a,b}$ and $K_{F}^{c_1,c_2,a,b}$. Thus, it suffices to treat nonzero positions $(p,q)$ for which we consider four cases.

\noindent \textbf{Case 1:} $(p,q)$ lies in the upper-left block $M_{F_{\pi},c_1}$ of $K$. Then
\[
K_{F_{\pi}}^{c_1,c_2,a,b}[p,q]=M_{F_{\pi},c_1}[p,q]
=Z_{\row_{c_1,p},\,F_{\pi}[\row_{c_1,q},c_1]}
=Z_{\row_{c_1,p},\,F_{\pi}[\zeta(q)]}.
\]
On the right-hand side,
\[
K_{F}^{c_1,c_2,a,b}[p,\psi(\pi)(q)]=M_{F,c_1}[p,\psi(\pi)(q)]
=Z_{\row_{c_1,p},\,F[\row_{c_1,\psi(\pi)(q)},c_1]}
=Z_{\row_{c_1,p},\,F[\zeta(\psi(\pi)(q))]}.
\]

\noindent \textbf{Case 2:} $(p,q)$ lies in the upper-right block $\overline M^{\,b}_{F_{\pi},c_2}$ in $K$. Then
\[
K_{F_{\pi}}^{c_1,c_2,a,b}[p,q]=\overline M^{\,b}_{F_{\pi},c_2}[p,q]
=Z_{\row_{c_1,p},\,F_{\pi}[\row_{c_2,q},c_2]}
=Z_{\row_{c_1,p},\,F_{\pi}[\zeta(q)]},
\]
and
\[
K_{F}^{c_1,c_2,a,b}[p,\psi(\pi)(q)]=\overline M^{\,b}_{F,c_2}[p,\psi(\pi)(q)]
=Z_{\row_{c_1,p},\,F[\row_{c_2,\psi(\pi)(q)},c_2]}
=Z_{\row_{c_1,p},\,F[\zeta(\psi(\pi)(q))]}.
\]

\noindent \textbf{Case 3:} $(p,q)$ lies in the lower-left block $\overline M^{\,a}_{F_{\pi},c_1}$. Then
\[
K_{F_{\pi}}^{c_1,c_2,a,b}[p,q]=\overline M^{\,a}_{F_{\pi},c_1}[p,q]
=Z_{\row_{c_2,p},\,F_{\pi}[\row_{c_1,q},c_1]}
=Z_{\row_{c_2,p},\,F_{\pi}[\zeta(q)]},
\]
while
\[
K_{F}^{c_1,c_2,a,b}[p,\psi(\pi)(q)]=\overline M^{\,a}_{F,c_1}[p,\psi(\pi)(q)]
=Z_{\row_{c_2,p},\,F[\row_{c_1,\psi(\pi)(q)},c_1]}
=Z_{\row_{c_2,p},\,F[\zeta(\psi(\pi)(q))]}.
\]

\noindent \textbf{Case 4:} $(p,q)$ lies in the lower-right block $M_{F_{\pi},c_2}$. Then
\[
K_{F_{\pi}}^{c_1,c_2,a,b}[p,q]=M_{F_{\pi},c_2}[p,q]
=Z_{\row_{c_2,p},\,F_{\pi}[\row_{c_2,q},c_2]}
=Z_{\row_{c_2,p},\,F_{\pi}[\zeta(q)]},
\]
and
\[
K_{F}^{c_1,c_2,a,b}[p,\psi(\pi)(q)]=M_{F,c_2}[p,\psi(\pi)(q)]
=Z_{\row_{c_2,p},\,F[\row_{c_2,\psi(\pi)(q)},c_2]}
=Z_{\row_{c_2,p},\,F[\zeta(\psi(\pi)(q))]}.
\]

In all four cases, equality follows by \eqref{eq:filling-transport} which we prove next.

We partition the column-index set $[\ell'_{c_1}+\ell'_{c_2}]$ as
\[
Q_{\mathrm{L}}:=\{1,\dots,\ell'_{c_1}-a\},\qquad
Q_{\mathrm{M}}:=\{\ell'_{c_1}-a+1,\dots,\ell'_{c_1}+b\},\qquad
Q_{\mathrm{R}}:=\{\ell'_{c_1}+b+1,\dots,\ell'_{c_1}+\ell'_{c_2}\},
\]
and consider two cases for the column index $q$.

\noindent \textbf{Case 1:} $q\in Q_{\mathrm{L}}\cup Q_{\mathrm{R}}$. Then $\zeta(q)$ lies in column $c_1$ above the bottom $a$ cells or in column $c_2$ below the top $b$ cells. Hence, $\zeta(q) \notin \Gar_{c_1,c_2}^{a,b}$. Since the $\mathfrak{S}_{a+b}^{a,b}$-action leaves the entries at positions outside $\Gar_{c_1,c_2}^{a,b}$ unchanged, we have
\[
F_{\pi}[\zeta(q)]=F[\zeta(q)].
\]
Moreover, by the definition of $\psi$, the permutation $\psi(\pi)$ fixes all indices outside $Q_{\mathrm{M}}$, so $\psi(\pi)(q)=q$ and thus $\zeta(\psi(\pi)(q))=\zeta(q)$. Consequently,
\[
F_{\pi}[\zeta(q)] = F[\zeta(q)] = F[\zeta(\psi(\pi)(q))],
\]
which proves \eqref{eq:filling-transport} for all $q\in Q_{\mathrm{L}}\cup Q_{\mathrm{R}}$.

\noindent \textbf{Case 2:} $q\in Q_{\mathrm{M}}=\{\ell'_{c_1}-a+1,\dots,\ell'_{c_1}+b\}$. Then $\zeta(q)\in \Gar_{c_1,c_2}^{a,b}$. Let $\eta:\Gar_{c_1,c_2}^{a,b}\to\{1,\dots,a+b\}$ be the bijection introduced in Subsection~\ref{subsec:garnirdef}. By Definition~\ref{def:permutation_filling},
\[
F_{\pi}[\zeta(q)] = F\left[\eta^{-1} \big(\pi(\eta(\zeta(q)))\big)\right].
\]
Define $\eta':Q_{\mathrm{M}}\to\{1,\dots,a+b\}$ by $\eta'(q):=q-(\ell'_{c_1}-a)$. From the definition of $\psi$,
\begin{equation}\label{eq:sigma_q}
\psi(\pi)(q)  =  (\eta')^{-1}\!\big(\pi(\eta'(q))\big).
\end{equation}
For any $q\in Q_{\mathrm{M}}$ one verifies
\begin{equation}\label{eq:zeta_q}
\eta(\zeta(q))=\eta'(q)\quad\text{and}\quad \zeta\circ(\eta')^{-1}=\eta^{-1}.
\end{equation}
Hence
\[
\zeta(\psi(\pi)(q))
= \zeta \left((\eta')^{-1} \big(\pi(\eta'(q))\big)\right)
= \zeta \left((\eta')^{-1} \big(\pi(\eta(\zeta(q)))\big)\right)
= \eta^{-1} \big(\pi(\eta(\zeta(q)))\big),
\]
where the first equality uses \eqref{eq:sigma_q}, the second uses the first identity in \eqref{eq:zeta_q}, and the third uses the second identity in \eqref{eq:zeta_q}. Therefore $F_{\pi}[\zeta(q)] = F\left[\eta^{-1} \big(\pi(\eta(\zeta(q)))\big)\right] = F[\zeta(\psi(\pi)(q))]$ for all $q\in Q_{\mathrm{M}}$, completing the proof of \eqref{eq:filling-transport}.
  \end{proof}

\begin{lemma}\label{lemma:garnir_determinant_relation}
  Fix a filling $F$ of shape $\lambda / \mu$. Choose $c_1, c_2 \in [\lambda_1]$ such that $c_1 < c_2$ and $a, b \in \mathbb{N}$ such that $a \leq \ell'_{c_1}$ and $b \leq \ell'_{c_2}$. If $\lambda_{c_1}'-a < \mu_{c_2}'+b$, then
\[
  \sum_{\pi\in \mathfrak{S}_{a+b}^{a, b}}\sign{\pi}D_{F_{\pi}, c_1}D_{F_{\pi}, c_2} = 0.
\]
\end{lemma}

\begin{proof}
Let $K=K_{F}^{c_1, c_2,a,b}$ and for $\pi\in \mathfrak{S}_{a+b}^{a, b}$ let $K_{F_{\pi}}=K_{F_{\pi}}^{c_1, c_2,a,b}$. Let $S$ be defined as in Lemma \ref{lemma:bijection}. Then $\lambda_{c_1}'-a < \mu_{c_2}'+b$ implies, by Lemma \ref{lemma:singularity}, that $K$ is singular. Hence,
\begin{align*}
0 = \det(K) &= \sum_{\pi\in \mathfrak{S}_{\ell'_{c_1}+\ell'_{c_2}}}\text{sgn}(\pi)\prod_{i=1}^{\ell'_{c_1}+\ell'_{c_2}} K[i, \pi(i)] \\
& = \sum_{\sigma\in S}\text{sgn}(\sigma)\prod_{i=1}^{\ell'_{c_1}+\ell'_{c_2}} K[i, \sigma(i)]  +  \sum_{\sigma\in \mathfrak{S}_{\ell'_{c_1}+\ell'_{c_2}}\setminus S}\text{sgn}(\sigma)\prod_{i=1}^{\ell'_{c_1}+\ell'_{c_2}} K[i, \sigma(i)].
\end{align*}
By Lemma \ref{lemma:zero}, if $\sigma\in \mathfrak{S}_{\ell'_{c_1}+\ell'_{c_2}}\setminus S$, then there exists some $i$ such that $K[i, \sigma(i)] = 0$. Thus, every term in the summation indexed by $\mathfrak{S}_{\ell'_{c_1}+\ell'_{c_2}}\setminus S$ is $0$.

Accordingly, we have
\begin{equation} \label{eq:singKImplication}
  0 = \det(K) = \sum_{\sigma\in S}\text{sgn}(\sigma)\prod_{i=1}^{\ell'_{c_1}+\ell'_{c_2}} K[i, \sigma(i)].
\end{equation}
Similarly, for $\pi\in \mathfrak{S}_{a+b}^{a, b}$ we have via the Leibniz expansion of the determinant that
\[
D_{F_{\pi}, c_1} = \sum_{\sigma_1\in \mathfrak{S}_{\ell'_{c_1}}}\sign{\sigma_1}\prod_{i=1}^{\ell'_{c_1}} M_{F_{\pi}, c_1}[i, \sigma_1(i)],\quad
D_{F_{\pi}, c_2} = \sum_{\sigma_2\in \mathfrak{S}_{\ell'_{c_2}}}\sign{\sigma_2}\prod_{i=1}^{\ell'_{c_2}} M_{F_{\pi}, c_2}[i, \sigma_2(i)].
\]
Given that $(\sigma_1, \sigma_2)\in \mathfrak{S}_{\ell'_{c_1}}\times \mathfrak{S}_{\ell'_{c_2}}$, we let
$\overline{\sigma}:=\varphi((\sigma_1, \sigma_2))$.
Then,
\begin{align*}
 & \sum_{\pi\in \mathfrak{S}_{a+b}^{a, b}}\sign{\pi}D_{F_{\pi}, c_1}D_{F_{\pi}, c_2}  \\
&= \sum_{\pi\in \mathfrak{S}_{a+b}^{a, b}}\text{sgn}(\pi)\left(\sum_{\sigma_1\in \mathfrak{S}_{\ell'_{c_1}}}\sign{\sigma_1}\prod_{i=1}^{\ell'_{c_1}} (M_{F_{\pi}, c_1})[i, \sigma_1(i)]\right)\left( \sum_{\sigma_2\in \mathfrak{S}_{\ell'_{c_2}}}\sign{\sigma_2}\prod_{i=1}^{\ell'_{c_2}} (M_{F_{\pi}, c_2})[i, \sigma_2(i)]\right) \\
& = \sum_{\pi\in \mathfrak{S}_{a+b}^{a, b}}\text{sgn}(\pi)\left(\sum_{(\sigma_1, \sigma_2)\in \mathfrak{S}_{\ell'_{c_1}}\times \mathfrak{S}_{\ell'_{c_2}}}\sign{(\sigma_1, \sigma_2)}\left(\prod_{i=1}^{\ell'_{c_1}} M_{F_{\pi}, c_1}[i, \sigma_1(i)]\right)\left(\prod_{i=1}^{\ell'_{c_2}} M_{F_{\pi}, c_2}[i, \sigma_2(i)]\right)\right)\\
&= \sum_{\pi \in \mathfrak{S}_{a+b}^{\,a,b}}
   \sign{\pi}
   \left(
     \sum_{\overline{\sigma}\in
           \varphi\!\left(\mathfrak{S}_{\ell'_{c_1}}
           \times \mathfrak{S}_{\ell'_{c_2}}\right)}
     \sign{\overline{\sigma}}
     \left(
       \prod_{i=1}^{\ell'_{c_1}}
         K_{F_{\pi}}[i,\overline{\sigma}(i)]
     \right)
     \left(
       \prod_{i=\ell'_{c_1}+1}^{\ell'_{c_1}+\ell'_{c_2}}
         K_{F_{\pi}}[i,\overline{\sigma}(i)]
     \right)
   \right).
\end{align*}
For any $\pi\in \mathfrak{S}_{a+b}^{a, b}$, we have the relationship
$K_{F_{\pi}}[p, q] = K[p, \psi(\pi)(q)]$ by Lemma \ref{lemma:matrix_iden}. Hence,
\[
  \left(
    \prod_{i=1}^{\ell'_{c_1}}
      K_{F_{\pi}}[i,\overline{\sigma}(i)]
  \right)
  \left(
    \prod_{i=\ell'_{c_1}+1}^{\ell'_{c_1}+\ell'_{c_2}}
      K_{F_{\pi}}[i,\overline{\sigma}(i)]
  \right)=
  \prod_{i=1}^{\ell'_{c_1}+\ell'_{c_2}} K[\,i,(\psi(\pi)\!\cdot\!\overline{\sigma})(i)\,].
\]
Using this identity we can continue our above equation yielding
\begin{align*}
\sum_{\pi\in \mathfrak{S}_{a+b}^{a, b}}\sign{\pi}D_{F_{\pi}, c_1}D_{F_{\pi}, c_2} &= \sum_{\pi \in \mathfrak{S}_{a+b}^{\,a,b}}
   \sign{\pi}
   \left(
     \sum_{\overline{\sigma}\in
     \varphi\!\left(\mathfrak{S}_{\ell'_{c_1}}
     \times \mathfrak{S}_{\ell'_{c_2}}\right)}
     \sign{\overline{\sigma}}
     \prod_{i=1}^{\ell'_{c_1}+\ell'_{c_2}}
       K\!
       [\,i,(\psi(\pi)\!\cdot\!\overline{\sigma})(i)\,]
   \right) \\[4pt]
&= \sum_{\pi \in \mathfrak{S}_{a+b}^{\,a,b}}
   \left(
     \sum_{\overline{\sigma}\in
           \varphi\!\left(\mathfrak{S}_{\ell'_{c_1}}
           \times \mathfrak{S}_{\ell'_{c_2}}\right)}
     \operatorname{sgn}\!\left(\psi(\pi)\!\cdot\!\overline{\sigma}\right)
     \prod_{i=1}^{\ell'_{c_1}+\ell'_{c_2}}
       K\!
       [\,i,(\psi(\pi)\!\cdot\!\overline{\sigma})(i)\,]
   \right)\\
&=   \sum_{\sigma\in S}
     \sign{\sigma}
     \prod_{i=1}^{\ell'_{c_1}+\ell'_{c_2}}
       K[\,i,\sigma(i)\,] =0.
\end{align*}
where the second equality follows from the fact that $\psi$ is a sign-preserving homomorphism and $\sign{}$ is a group homomorphism, the third equality follows from the sign-preserving bijection established in Lemma \ref{lemma:bijection}, and the final equality from \eqref{eq:singKImplication}.
\end{proof}

\begin{corollary}\label{cor:garnir_determinant_relation}
  Fix a filling $F$ of shape $\lambda / \mu$. Choose $c_1, c_2 \in [\lambda_1]$ such that $c_1 < c_2$ and $a, b \in \mathbb{N}$ such that $a \leq \ell'_{c_1}$ and $b \leq \ell'_{c_2}$. If $\lambda_{c_1}'-a < \mu_{c_2}'+b$, then
  \[
    \sum_{\pi\in \mathfrak{S}_{a+b}^{a, b}}\sign{\pi}D_{F_{\pi}} =0.
  \]
\end{corollary}

\begin{proof}
By definition, $D_{F_{\pi}} = \prod_{c} D_{F_{\pi}, c}$. The action of the shuffle $\pi \in \mathfrak{S}_{a+b}^{a, b}$ is, by construction, restricted to modifying the entries of the filling $F$ within columns $c_1$ and $c_2$. For any column $c \notin \{c_1, c_2\}$, the filling $F_{\pi}$ is identical to $F$ in that column, which implies that the corresponding determinant factors are unchanged, yielding $D_{F_{\pi}, c} = D_{F, c}$.

Thus
\begin{align*}
\sum_{\pi\in \mathfrak{S}_{a+b}^{a, b}}\sign{\pi}D_{F_{\pi}} &= \sum_{\pi\in \mathfrak{S}_{a+b}^{a, b}}\sign{\pi}\left(\prod_{c} D_{F_{\pi}, c}\right) \\
&= \left(\prod_{c \neq c_1, c_2}D_{F, c}\right) \left(\sum_{\pi\in \mathfrak{S}_{a+b}^{a, b}}\sign{\pi} D_{F_{\pi}, c_1}D_{F_{\pi}, c_2}\right).
\end{align*}
The second factor equals $0$ by Lemma~\ref{lemma:garnir_determinant_relation}, yielding our desired identity.
\end{proof}

\subsection{The key homomorphism and a basis for the skew Schur module}
Fix a skew partition $\lambda/\mu$, and fix an ordered basis $(e_1,\ldots,e_m)$ of the free $R$-module $E$, so all filling entries lie in $[m]=\{1,\ldots,m\}$. Let $Z_{i,j}$ for $1 \leq i \leq \lambda'_1$ and $1 \leq j \leq m$ be a set of indeterminates, and $R[Z]$ the polynomial ring over $R$ in these indeterminates. In this subsection we construct an $R$-linear homomorphism $\Phi:E^{\lambda/\mu}\to R[Z]$ by sending a filling $F$ to the determinantal polynomial $D_F$. The determinantal Garnir identities proved above ensure that $\Phi$ descends to the skew Schur module. As a first application, we use $\Phi$ to prove linear independence of the family \[\{\bare{T} := e_T\!+\!Q \in M_R/Q:\,T\in \mathrm{SSYT}(\lambda/\mu,z)\text{ for some content }z\},\] so that semistandard tableaux index a basis. The remainder of the subsection shows that these vectors span the module. This homomorphism $\Phi$ will also be crucial to the proof of our non-iterative straightening method.

\begin{proposition}\label{prop:lemma3inFulton}
  Fix a skew partition $\lambda/\mu$. Let $E$ be a free $R$-module of rank $m$. There exists a unique $R$-module homomorphism $\Psi\!\!:\!\!E^{\lambda/\mu}\!\to\!R[Z]$ which, under the identification of $E^{\lambda/\mu}$ with $M_R/Q$, maps $\bare{F} := e_F\!+\!Q$ to $D_F$.
\end{proposition}

\begin{proof}
  Recall that $M_R$ is the free $R$-module with basis $\{e_F \mid F \text{ is a filling of } \lambda/\mu\}$. We define a map $\Phi: M_R \to R[Z]$ by specifying that
  \[
    \Phi(e_F) = D_F,
  \]
  and extending linearly to all of $M_R$.  This map is by construction an $R$-module homomorphism. To show that $\Phi$ descends to a well-defined homomorphism on the quotient $M_R/Q$, we must show that $\Phi$ maps each of the three types of generators of the submodule $Q$ to zero.

  \noindent \textbf{Case 1:} generator $e_F$ where $F$ is a filling that has two identical entries in the same column $c$ of $F$. The case hypothesis implies that two columns of the matrix $M_{F,c}$ are identical. The determinant of a matrix with two identical columns is zero over any base ring $R$, so $D_{F,c} = \det(M_{F,c}) = 0$. Consequently, $\Phi(e_F) = D_F = \prod_{c=1}^{\lambda_1} D_{F,c} = 0$.

  \noindent \textbf{Case 2:}  generator $e_F + e_{F'}$ where $F$ is a filling and $F'$ is obtained from $F$ by interchanging two entries in some column $c$ of $F$. The corresponding matrix $M_{F',c}$ is obtained from $M_{F,c}$ by swapping two columns. This changes the sign of the determinant, so $D_{F',c} = -D_{F,c}$. For any other column $j \neq c$, $D_{F',j} = D_{F,j}$. Therefore, $D_{F'} = -D_F$ and $\Phi(e_F + e_{F'}) = \Phi(e_F) + \Phi(e_{F'}) = D_F + D_{F'} = 0$.

  \noindent \textbf{Case 3:}  generator $\sum_{\pi\in \mathfrak{S}_{a+b}^{a,b}} \sign{\pi} e_{F_{\pi}}$, where $F$ is a filling and parameters $(c_1, c_2, a, b)$ are $\lambda/\mu$-admissible. Applying $\Phi$ to this generator gives
    \[
      \Phi\left(\sum_{\pi \in \mathfrak{S}_{a+b}^{a,b}} \sign{\pi} e_{F_{\pi}}^{c_1, c_2}\right) = \sum_{\pi \in \mathfrak{S}_{a+b}^{a,b}} \sign{\pi} \Phi(e_{F_{\pi}}) = \sum_{\pi \in \mathfrak{S}_{a+b}^{a,b}} \sign{\pi} D_{F_{\pi}} = 0,
    \]
    where the final equality is Corollary~\ref{cor:garnir_determinant_relation}.

  Since $\Phi$ annihilates all generators of $Q$, we have $Q \subseteq \ker(\Phi)$. By the universal property of quotient modules, the map $\Phi$ descends to a unique $R$-module homomorphism $\Psi: E^{\lambda/\mu} \to R[Z]$ that maps $\bare{F}$ to $D_F$.
\end{proof}

Proposition~\ref{prop:lemma3inFulton} is now used to construct a basis for the skew Schur module.

\begin{definition}\label{def:lexicographical_ordering_on_monomials}
We establish a total order on the set of indeterminates $\{Z_{i,j}\}$ by defining $Z_{i,j} < Z_{i',j'}$ if the index pair $(i,j)$ precedes $(i',j')$ in lexicographical order. This variable ordering induces the lexicographical order, which is a total order, on the set of all monomials in the polynomial ring $R[Z]$. To compare any two distinct monomials $M_1$ and $M_2$, we identify the smallest indeterminate $Z_{i_0,j_0}$ in the established variable ordering for which the exponent in $M_1$ differs from its exponent in $M_2$. We then define $M_1 < M_2$ if the exponent of $Z_{i_0,j_0}$ in $M_1$ is less than its exponent in $M_2$.
\end{definition}

\begin{example}
  Consider monomials $M_1 = Z_{1,2}^3 = Z_{1,1}^0 Z_{1,2}^3$ and $M_2 = Z_{1,1} = Z_{1,1}^1 Z_{1,2}^0$ in $R[Z]$ with indeterminates $Z_{1,1}, Z_{1,2}$. The exponent of the smallest variable $Z_{1,1}$ in $M_1$ is $0$, while the exponent of $Z_{1,1}$ in $M_2$ is $1$. These exponents differ and $0 < 1$, so $M_1 < M_2$.
\end{example}

\begin{lemma}\label{lem:property_of_monomial_ordering}
  Suppose $M_1, M_2, N_1, N_2$ are monomials satisfying $M_1 < M_2$ and $N_1 \leq N_2$. Then their products satisfy $M_1N_1 < M_2N_2$.
  \end{lemma}
  \begin{proof} The lexicographical order as defined in Definition~\ref{def:lexicographical_ordering_on_monomials} is a monomial order (see, e.g., \cite[Chapter~2, \S~2, Proposition 4]{CoxLittleOshea}). One of the defining axioms of a monomial order is that for monomials $A, B,$ and $C$, if $A < B$, then $AC < BC$.

Given the hypotheses $M_1 < M_2$ and $N_1 \le N_2$, we apply this axiom to obtain the inequalities $M_1 N_1 < M_2 N_1$ and $M_2 N_1 \le M_2 N_2$. The desired result follows by transitivity.
  \end{proof}

\begin{lemma}\label{lem:largest_monomial_determinant}
  Let $T$ be a SSYT with shape $\lambda/\mu$ and let $c \in [\lambda_1]$. The largest monomial in $D_{T,c}$ is given by the product of the diagonal entries of $M_{T,c}$, namely
  \[
  \prod_{i=1}^{\ell_c'} M_{T,c}[i,i].
  \]
  \end{lemma}

  \begin{proof}
  Given a semistandard Young tableaux $T$ and $c\in [\lambda_1]$, we have
  \[
  D_{T,c} = \det(M_{T,c}) = \sum_{\sigma \in \mathfrak{S}_{\ell_c'}} \prod_{i=1}^{\ell_c'} M_{T,c}[i,\sigma(i)].
  \]
  Consider the monomial
  \[
  \prod_{i=1}^{\ell_c'} M_{T,c}[i,i] = Z_{\mu_c'+1,a_1} \cdots Z_{\lambda_c',a_{\ell_c'}}.
  \]
  The indeterminates in this product are strictly ordered, $Z_{\mu_c'+1,a_1} < Z_{\mu_c'+2,a_2} < \dots < Z_{\lambda_c',a_{\ell_c'}}$, since the first indices, $\mu_c'+1, \mu_c'+2, \dots, \lambda_c'$, are strictly increasing. The semi-standard property of $T$ implies that its entries increase strictly down each column, and so $a_1 < a_2 < \dots < a_{\ell_c'}$.

  Let $M$ be any monomial in $D_{T,c}$ different from $\prod_{i=1}^{\ell_c'} M_{T,c}[i,i]$. Then there exists $\sigma \in \mathfrak{S}_{\ell_c'}$ such that
  \[
  M = Z_{\mu_c'+1,b_1} \cdots Z_{\lambda_c',b_{\ell_c'}},
  \]
  where $b_i = a_{\sigma(i)}$. Because $M \neq \prod_{i=1}^{\ell_c'} M_{T,c}[i,i]$, there is a smallest index $k$ with $b_k \neq a_k$.

  Let $k$ be the smallest index with $b_k \neq a_k$. Suppose, for contradiction, that $b_k < a_k$. Then $\sigma(k) < k$. For every $j < k$ we have $b_j = a_j$ by minimality, hence $\sigma(j) = j$. In particular, $\sigma(\sigma(k)) = \sigma(k)$ with $\sigma(k) \neq k$, which contradicts injectivity of $\sigma$. Therefore $b_k > a_k$.

  Thus, the smallest indeterminate that differs between $M$ and $\prod_{i=1}^{\ell_c'} M_{T,c}[i,i]$ is $Z_{\mu_c'+k,a_k}$, and it appears in $\prod_{i=1}^{\ell_c'} M_{T,c}[i,i]$ not $M$. Thus $M\!<\!\prod_{i=1}^{\ell_c'} M_{T,c}[i,i]$ in the lexicographic order.
  \end{proof}

\begin{corollary}\label{cor:largest_monomial_of_D_T}
  Let $T$ be a SSYT with shape $\lambda/\mu$. Then $\prod_{c=1}^{\lambda_1}\prod_{i=1}^{\ell_{c}'} M_{T, c}[i, i]$ is the unique largest monomial in $D_T$. In other words, the largest monomial in $D_T$ is the product of the diagonal entries of each $M_{T, c}$ and it has coefficient $1$.
\end{corollary}

\begin{proof}
  By Lemma~\ref{lem:largest_monomial_determinant}, for each column $c \in [\lambda_1]$, the largest monomial in the determinant $D_{T,c}$ is the diagonal monomial, $\prod_{i=1}^{\ell_c'} M_{T,c}[i,i]$. Since $D_T = \prod_{c\in [\lambda_1]} D_{T, c}$, repeated applications of Lemma \ref{lem:property_of_monomial_ordering} implies that the product of these largest monomials is the largest monomial in the full product. Thus $\prod_{c\in [\lambda_1]} \prod_{i=1}^{\ell_c'} M_{T, c}[i, i]$ is the largest monomial in $D_T$.
\end{proof}

We now define a total order on the set of SSYT of the same shape. This order is constructed to be compatible with the lexicographical order on the leading monomials of the associated polynomials $D_T$. This compatibility, proven in Lemma~\ref{lem:correspondence_of_orderings}, is the central component in the subsequent proof of linear independence.

\begin{definition}\label{def:reading_word_order}
Fix the ordered alphabet $[m]=\{1<2<\cdots<m\}$. For a filling $F$ of shape $\lambda/\mu$, the \emph{reading word} is the word $\mathrm{rw}(F)\in [m]^{|\lambda/\mu|}$ obtained by reading the entries of $F$ along each row from left to right, starting with the top row and proceeding to the bottom row. Equip $[m]^{|\lambda/\mu|}$ with the induced lexicographic order. For fillings $F,F'$ of shape $\lambda/\mu$, define a total order by declaring that $F \prec F'$ if and only if $\mathrm{rw}(F)$ is lexicographically smaller than $\mathrm{rw}(F')$.
\end{definition}

For any filling $T$ of shape $\lambda/\mu$, any row index $r$, and any entry $j$, let $m_T(r,j)$ denote the number of occurrences of the entry $j$ in row $r$ of $T$.

\begin{example}\label{ex:tableau_ordering}
Consider two semistandard Young tableaux $T,T'$ of shape $(4,3,2)/(1,1)$ with
\[
  T = \begin{ytableau}
    \none & 1 & 2 & 3 \\
    \none & 2 & 4 \\
    1 & 5
  \end{ytableau}
  \qquad \text{and} \qquad
  T' = \begin{ytableau}
    \none & 1 & 2 & 3 \\
    \none & 2 & 5 \\
    1 & 4
  \end{ytableau}.
\]
Then $\mathrm{rw}(T)=1232415$ and $\mathrm{rw}(T')=1232514$, and since $1232415$ is lexicographically smaller than $1232514$, we have $T\prec T'$. The first row where the contents of $T$ and $T'$ differ is row $2$ and $m_T(2,4)=1$ and $m_{T'}(2,4)=0$. For every $j<4$ the counts agree in that row. Thus in this instance $T\prec T'$ corresponds to the inequality $m_T(2,4)>m_{T'}(2,4)$. The following lemma establishes that this relationship holds in general.
\end{example}

\begin{lemma}\label{lem:tableau_ordering_by_entry_counts}
Let $T_1$ and $T_2$ be two distinct SSYT of shape $\lambda/\mu$. Let $r$ be the minimal row index such that the $r$-th rows of $T_1$ and $T_2$ differ. Let $j$ be the smallest entry for which $m_{T_1}(r,j)\neq m_{T_2}(r,j)$. Then $T_1\prec T_2$ if and only if $m_{T_1}(r,j)>m_{T_2}(r,j)$.
\end{lemma}

\begin{proof}
Let $A$ and $B$ be the subwords corresponding to row $r$ of $\mathrm{rw}(T_1)$ and $\mathrm{rw}(T_2)$, respectively. These are weakly increasing. By choice of $r$, all earlier rows coincide, so $T_1\prec T_2$ if and only if $A$ is lexicographically smaller than $B$.

Let $j$ be the smallest entry with $m_{T_1}(r,j)\neq m_{T_2}(r,j)$. For every $t<j$ the counts agree, hence the initial segments of $A$ and $B$ consisting of entries less than $j$ are identical. If $m_{T_1}(r,j)>m_{T_2}(r,j)$, then at the first position after these initial segments, $A$ has the letter $j$ while $B$ has a letter strictly larger than $j$, so $A<B$ in lexicographic order. Conversely, if $A<B$, then at the first differing position $A$ has some letter $x$ and $B$ has a letter $y>x$; by minimality of $j$ one must have $x=j$, and this forces $m_{T_1}(r,j)>m_{T_2}(r,j)$. Hence $T_1\prec T_2$ if and only if $m_{T_1}(r,j)>m_{T_2}(r,j)$.
\end{proof}

\begin{lemma}\label{lem:correspondence_of_orderings}
Let $T_1$ and $T_2$ be two distinct SSYT of shape $\lambda/\mu$. If $T_1 \prec T_2$, then the largest monomial of $D_{T_1}$ is strictly greater than the largest monomial of $D_{T_2}$.
\end{lemma}

\begin{proof}
  Let the largest monomials of $D_{T_1}$ and $D_{T_2}$ be $M_1$ and $M_2$, respectively. By Corollary~\ref{cor:largest_monomial_of_D_T},
\[
M_1=\prod_{c=1}^{\lambda_1}\prod_{i=1}^{\ell'_c} M_{T_1,c}[i,i],\qquad
M_2=\prod_{c=1}^{\lambda_1}\prod_{i=1}^{\ell'_c} M_{T_2,c}[i,i].
\]
For $M_1$, fix a column $c$ and an index $1\le i\le \ell'_c$. Then,
\[
M_{T_1,c}[i,i]=Z_{\row_{c,i},\,T_1[\row_{c,i},c]}.
\]
Thus, as $c$ ranges over columns and $i$ over $1,\dots,\ell'_c$, each occurrence of the entry $j$ in diagram row $r=\row_{c,i}$ of $T_1$ contributes one factor $Z_{r,j}$ to $M_1$. For any fixed diagram row $r$, the exponent of $Z_{r,j}$ in $M_1$ is therefore the number of columns $c$ with $T_1[r,c]=j$, that is, $m_{T_1}(r,j)$. The same reasoning applies to $M_2$. Thus,
\[
M_1 = \prod_{r=1}^{\lambda'_1}\ \prod_{j=1}^{m} Z_{r,j}^{\,m_{T_1}(r,j)},\qquad
M_2 = \prod_{r=1}^{\lambda'_1}\ \prod_{j=1}^{m} Z_{r,j}^{\,m_{T_2}(r,j)}.
\]

By Lemma~\ref{lem:tableau_ordering_by_entry_counts}, there exist a minimal row index $r$ and a smallest entry $j$ with $m_{T_1}(r,j)>m_{T_2}(r,j)$ and $m_{T_1}(r',j')=m_{T_2}(r',j')$ for all $(r',j')$ preceding $(r,j)$ in the lexicographic order. Since variables are ordered lexicographically, $(r,j)$ is the first index at which the exponent vectors of $M_1$ and $M_2$ differ, and the exponent of $Z_{r,j}$ in $M_1$ is larger than in $M_2$. Hence $M_2<M_1$ in lexicographic order.
\end{proof}

\begin{proposition}\label{prop:linear_independece_of_semi_standard_young_tableau}
Let $E$ be a free $R$-module of rank $m$. For a fixed skew shape $\lambda/\mu$ and content $z$, the set of elements $\{\bare{T} \mid T \in \operatorname{SSYT}(\lambda/\mu, z)\}$ is linearly independent in $E^{\lambda/\mu}$.
\end{proposition}
\begin{proof}
Consider a linear relation $\sum_{T \in \operatorname{SSYT}(\lambda/\mu, z)} c_T\,\bare{T}=0$ with $c_T\in R$, and order $\text{SSYT}(\lambda/\mu,z)$ as $T_1\prec T_2\prec\cdots\prec T_n$. Applying $\Psi$ from Proposition~\ref{prop:lemma3inFulton} and using linearity gives the single identity
\[
\Psi\!\left(\sum_{i=1}^{n} c_{T_i}\,\bare{T_i}\right)
\;=\;
\sum_{i=1}^{n} c_{T_i}\,\Psi(\bare{T_i})
\;=\;
\sum_{i=1}^{n} c_{T_i}\,D_{T_i}
\;=\;
0.
\]
For each $i$, let $M_i$ denote the largest monomial of $D_{T_i}$. By Corollary~\ref{cor:largest_monomial_of_D_T}, the coefficient of $M_i$ in $D_{T_i}$ is $1$. By Lemma~\ref{lem:correspondence_of_orderings}, the leading monomials satisfy $M_1\!>\!M_2\!>\!\cdots\!>\!M_n$. The coefficient of $M_1$ in the displayed identity is $c_{T_1}$, since $M_1$ does not occur in any $D_{T_j}$ for $j\!>\!1$, and so $c_{T_1}\!=\!0$. Removing that term and repeating the argument with $M_2$, then $M_3$, and so on, yields $c_{T_i}\!=\!0$ for all $i$. Thus, $\{\bare{T} \!\mid\! T\!\in\! \text{SSYT}(\lambda/\mu,z)\}$ is linearly independent in $E^{\lambda/\mu}$.
\end{proof}

\begin{definition}\label{def:column_words_order}
We define a total order, called the \emph{column word order}, on $F(\lambda/\mu,z)$. For a filling $F$ of shape $\lambda/\mu$, the \emph{column word} is the word $\mathrm{cw}(F)\in [m]^{|\lambda/\mu|}$ obtained by reading the entries of $F$ top to bottom within each column, starting with the rightmost column and proceeding leftward. For fillings $E,F\in F(\lambda/\mu,z)$, declare that $E \prec_{\mathrm{col}} F$ if and only if $\mathrm{cw}(E)$ is lexicographically smaller than $\mathrm{cw}(F)$. Define the operator $\operatorname{colsort}:F(\lambda/\mu,z)\to F(\lambda/\mu,z)$ by the rule that $\operatorname{colsort}(F)$ is the filling obtained from $F$ by sorting each column in weakly increasing order from top to bottom.
\end{definition}

\begin{example}
Let $F_1,F_2$ be fillings of shape $(3,2,1)/(1)$ with
\[
  F_1 = \begin{ytableau}
    \none & 3 & 4 \\
    2 & 2 \\
    1
  \end{ytableau}
  \qquad\text{and}\qquad
  F_2 = \begin{ytableau}
    \none & 2 & 4 \\
    1 & 3 \\
    2
  \end{ytableau}.
\]
Reading columns right-to-left, top-to-bottom, gives
\[
\mathrm{cw}(F_1)=43221,\qquad \mathrm{cw}(F_2)=42312.
\]
Since $42312$ is lexicographically smaller than $43221$, it follows that $F_2 \prec_{\mathrm{col}} F_1$.
\end{example}

\begin{lemma}\label{lem:garnir_on_column}
Let $F \in F(\lambda/\mu, z)$ be strictly increasing down each column. Suppose there exist $c_1<c_2$ and a row $r$ such that $(r,c_1),(r,c_2)\in D(\lambda/\mu)$ and $F(r,c_1)>F(r,c_2)$. Let $a=\lambda'_{c_1}-r+1$ and $b=r-\mu'_{c_2}$. Then $\lambda'_{c_1}-a<\mu'_{c_2}+b$, and for every $\pi\in \mathfrak{S}_{a+b}^{a,b}\setminus\{\mathrm{id}\}$ one of the following holds:
\begin{enumerate}
  \item $F_\pi$ has a repeated entry in some column, in which case $\bare{F_\pi}=0$ in $E^{\lambda/\mu}$.
  \item $F_\pi$ has no repeated entry in any column, and with $F':=\operatorname{colsort}(F_\pi)$ one has $F\prec_{\mathrm{col}} F'$.
\end{enumerate}
\end{lemma}

\begin{proof}
By the definitions of $a$ and $b$,
\[
\lambda'_{c_1}-a=r-1 \qquad\text{and}\qquad \mu'_{c_2}+b=r,
\]
so $\lambda'_{c_1}-a<\mu'_{c_2}+b$ holds. Fix $\pi\in \mathfrak{S}_{a+b}^{a,b}\setminus\{\mathrm{id}\}$. If $F_\pi$ has a repeated entry in some column, then $\bare{F_\pi}=0$ by the column-alternating property of $E^{\lambda/\mu}$. Hence assume $F_\pi$ has no repeated entry in any column.

Set $F':=\operatorname{colsort}(F_\pi)$ and define the multisets
\[
X_1=\{F[k,c_1]\!:\!\ r\le k\le \lambda'_{c_1}\}, \,\,
X_2=\{F[k,c_2]\!:\!\ \mu'_{c_2}<k\le r\}, \,\,
X_3=\{F[k,c_2]\!:\!\ r<k\le \lambda'_{c_2}\},
\]
and write $C_2=X_2\uplus X_3$ for the multiset of entries in column $c_2$ of $F$. Since $F$ is column-strict and $F(r,c_1)>F(r,c_2)$, every element of $X_1$ is strictly larger than every element of $X_2$. The $(a,b)$-shuffle replaces a nonempty submultiset $R\subseteq X_2$ by a multiset $S\subseteq X_1$ with $|S|=|R|$ and all elements of $S$ strictly larger than all elements of $X_2$. Set $X_2':=(X_2\setminus R)\uplus S$; then the multiset of entries in column $c_2$ of $F'$ is $C_2':=X_2'\uplus X_3$.

Let
\[
W=\operatorname{sort}(C_2)=(w_1\le\cdots\le w_N),\qquad
W'=\operatorname{sort}(C_2')=(w_1'\le\cdots\le w_N'),
\]
be the nondecreasing listings of $C_2$ and $C_2'$. Because $C_2'$ is obtained from $C_2$ by increasing some entries and never decreasing any, every order statistic is monotone: for all $k$, $w_k'\ge w_k$, and for at least one $k$ one has $w_k'>w_k$. Let $p$ be the minimal index with $w_p'>w_p$. Then $w_j'=w_j$ for all $j<p$ and $w_p'>w_p$, so $W'$ is lexicographically larger than $W$.

Finally, note that $W$ and $W'$ are precisely the subwords contributed by column $c_2$ to $\mathrm{cw}(F)$ and $\mathrm{cw}(F')$, respectively, and that columns strictly to the right of $c_2$ are unchanged by the shuffle. Hence $\mathrm{cw}(F')$ is lexicographically larger than $\mathrm{cw}(F)$, that is, $F\prec_{\mathrm{col}}F'$.
\end{proof}

\begin{lemma}\label{lem:row_straightening}
Let $F \in F(\lambda/\mu, z)$ be strictly increasing down each column. If $F$ is not a SSYT, then $\bare{F}$ can be expressed as a $R_0$-linear combination of $\bare{F'}$ such that $F' \in F(\lambda/\mu, z)$ and $F \! \prec_{col} \! F'$.
\end{lemma}
\begin{proof}
Since $F$ is not a SSYT, there exist columns $c_1<c_2$ and a row $r$ such that $(r,c_1),(r,c_2)\in D(\lambda/\mu)$ and $F(r,c_1)>F(r,c_2)$. Set $a =\lambda'_{c_1}-r+1$ and $b =r-\mu'_{c_2}$. Then $\lambda'_{c_1}-a=r-1$ and $\mu'_{c_2}+b=r$, so the Garnir inequality $\lambda'_{c_1}-a<\mu'_{c_2}+b$ holds. The Garnir relation in $E^{\lambda/\mu}$ gives
\[
\sum_{\pi\in \mathfrak{S}_{a+b}^{a,b}}\operatorname{sgn}(\pi)\,\bare{F_\pi}=0
\quad \text{ and thus }\quad
\bare{F} \;=\; -\sum_{\pi\in \mathfrak{S}_{a+b}^{a,b}\setminus\{\mathrm{id}\}}\operatorname{sgn}(\pi)\,\bare{F_\pi}.
\]
The shuffle preserves content, so each $F_\pi$ lies in $F(\lambda/\mu,z)$. Let $F_{\pi,c}:=\operatorname{colsort}(F_\pi)$. By the column-alternating property of $E^{\lambda/\mu}$, one has $\bare{F_\pi}=\pm \bare{F_{\pi,c}}$.

By Lemma~\ref{lem:garnir_on_column}, applied to the chosen $r,c_1,c_2$, for every nontrivial $\pi$ either $\bare{F_\pi}=0$ or $F\prec_{\mathrm{col}}F_{\pi,c}$. Substituting into the display above expresses $\bare{F}$ as a $R_0$-linear combination of $\bare{F'}$ with $F'\in F(\lambda/\mu,z)$ and $F\prec_{\mathrm{col}}F'$ for every nonzero term, as claimed.
\end{proof}

\begin{proposition}\label{prop:straightening}
Let $F \in F(\lambda/\mu, z)$ be a filling such that $\bare{F} \neq 0$ in $E^{\lambda/\mu}$. Then $\bare{F}$ can be expressed as a $R_0$-linear combination of elements $\{\bare{T}\}$ with $T \in \operatorname{SSYT}(\lambda/\mu, z)$.
\end{proposition}

\begin{proof}
If $F$ has a repeated entry in some column, then $\bare{F}=0$ by Lemma~\ref{lem:lemma1}(i), contrary to hypothesis. Thus the entries in each column of $F$ are distinct. Let $F_c:=\operatorname{colsort}(F)$. By Lemma~\ref{lem:lemma1}(ii), $\bare{F}=\pm \bare{F_c}$, so it suffices to straighten $\bare{F_c}$.

Proceed by strong induction on the set of column-strict fillings in $F(\lambda/\mu,z)$, descending by $\prec_{\mathrm{col}}$. Since the map $F\mapsto \mathrm{cw}(F)$ identifies $F(\lambda/\mu,z)$ with a finite subset of $[m]^{|\lambda/\mu|}$ ordered lexicographically, $\prec_{\mathrm{col}}$ is well-founded. Hence there are no infinite strictly increasing chains under $\prec_{\mathrm{col}}$, and the process terminates. If $F_c$ is maximal under $\prec_{\mathrm{col}}$, then $F_c$ must be an SSYT; otherwise there exist $r$ and $c_1<c_2$ with $F_c(r,c_1)>F_c(r,c_2)$, and Lemma~\ref{lem:row_straightening} would express $\bare{F_c}$ as a $R_0$-linear combination of $\bare{G'}$ with $F_c\prec_{\mathrm{col}}G'$, contradicting maximality.

If $F_c$ is not an SSYT, choose $r,c_1<c_2$ with $F_c(r,c_1)>F_c(r,c_2)$. By Lemma~\ref{lem:row_straightening},
\[
\bare{F_c} \;=\; \sum_i c_i\,\bare{G_i},
\]
where each $G_i$ is column-strict and $F_c\!\prec_{\mathrm{col}} \!G_i$. By the inductive hypothesis, each $\bare{G_i}$ is a $R_0$-linear combination of $\{\bare{T}\!:\! T \in \text{SSYT}(\lambda/\mu,z)\}$, and substituting yields the same for $\bare{F_c}$.

Thus every column-strict filling straightens to a $R_0$-linear combination of $\bare{T}$ with $T\in \text{SSYT}(\lambda/\mu,z)$, and the reduction $\bare{F}=\pm \bare{F_c}$ completes the argument for all $F\in F(\lambda/\mu,z)$.
\end{proof}

\begin{theorem}\label{thm:basis_of_schur_module}
If $E$ is a free $R$-module with basis $\{e_1, \ldots, e_m\}$, then the skew Schur module $E^{\lambda/\mu}$ is a free $R$-module. Its basis is given by the set of elements corresponding to SSYT,
\[
\mathcal{S}_{\lambda/\mu} := \bigcup_{z} \{\bare{T} \mid T \in \operatorname{SSYT}(\lambda/\mu, z)\},
\]
where the union is over all contents $z$ with support in $\{1, \ldots, m\}$.
\end{theorem}

\begin{proof}
We will show that the set $\mathcal{S}_{\lambda/\mu}$ is a basis for $E^{\lambda/\mu}$ by proving that it is a spanning set and is linearly independent.

We first show spanning. By Lemma~\ref{lem:lemma1}, the skew Schur module $E^{\lambda/\mu}$ is isomorphic to the quotient module $M_R/Q$, where $M_R$ is the free $R$-module with basis $\{e_F\}$ over all fillings $F$. The set of cosets $\{\bare{F}\}$ therefore spans $E^{\lambda/\mu}$. For any given filling $F$, its corresponding element $\bare{F}$ is either zero in the module or, by Proposition~\ref{prop:straightening}, can be expressed as a $R_0$-linear combination of elements from $\mathcal{S}_{\lambda/\mu}$. Thus, the set $\mathcal{S}_{\lambda/\mu}$ forms a spanning set for $E^{\lambda/\mu}$.

We now argue linear independence. The module $M_R$ decomposes into a direct sum over all possible contents $z$, $M_R = \bigoplus_z M_{R,z}$, where $M_{R,z}$ is the free submodule with basis $\{e_F \mid F \in F(\lambda/\mu, z)\}$. The relations that generate the submodule $Q$ are content-homogeneous; that is, they only relate fillings of the same content. Consequently, the submodule of relations also decomposes as a direct sum $Q = \bigoplus_z Q_z$, where $Q_z = Q \cap M_{R,z}$. The skew Schur module therefore has the direct sum decomposition
\[
E^{\lambda/\mu} \cong M_R/Q \cong \bigoplus_z (M_{R,z} / Q_z).
\]
To prove the linear independence of the spanning set $\mathcal{S}_{\lambda/\mu}$, consider a linear combination of its elements which equals zero in $E^{\lambda/\mu}$:
\[
\sum_{\bare{T} \in \mathcal{S}_{\lambda/\mu}} c_T \bare{T} = 0.
\]
Grouping the terms by content gives
\[
\sum_z \left(\sum_{T \in \operatorname{SSYT}(\lambda/\mu, z)} c_T \bare{T}\right) = 0.
\]
By the properties of a direct sum, an element is zero if and only if each of its components in the direct sum is zero. Therefore, for each content $z$, we must have
\[
\sum_{T \in \operatorname{SSYT}(\lambda/\mu, z)} c_T \bare{T} = 0
\]
in the submodule $M_{R,z}/Q_z$. By Proposition~\ref{prop:linear_independece_of_semi_standard_young_tableau}, for any fixed content $z$, the set $\{\bare{T} \mid T \in \operatorname{SSYT}(\lambda/\mu, z)\}$ is linearly independent. This implies that $c_T = 0$ for all $T \in \operatorname{SSYT}(\lambda/\mu, z)$. Since this holds for every content $z$, all coefficients in the original sum must be zero. Thus, the spanning set $\mathcal{S}_{\lambda/\mu}$ is linearly independent.

Since the set $\mathcal{S}_{\lambda/\mu}$ is a spanning set and is linearly independent, it forms a basis for $E^{\lambda/\mu}$. The existence of a basis implies that $E^{\lambda/\mu}$ is a free $R$-module.
\end{proof}

\section{Rearrangement Coefficients} \label{sec:rearrangement}
Throughout this section, fix a commutative ring $R$ and a free $R$-module $E$ of rank $m$ with ordered basis $\{e_1,\dots,e_m\}$. Let $R_0=\mathbb{Z}\cdot 1_R\subset R$ denote the prime subring. Fix partitions $\lambda$ and $\mu$ with $\mu\subseteq\lambda$, and consider the skew shape $\lambda/\mu$. Let $\lambda'$ and $\mu'$ denote the conjugate partitions. For each column index $c\in[\lambda_1]$, set $\ell'_c=\lambda'_c-\mu'_c$ and define the row-index sequence
\[
\row_c=(\row_{c,1},\ldots,\row_{c,\ell'_c})=(\mu'_c+1,\ldots,\lambda'_c),
\]
so that $\row_{c,i}$ is the row index of the $i$-th box (from top to bottom) in column $c$. Given a filling $F$ of shape $\lambda/\mu$ with entries in $\{1,\ldots,m\}$, expressing $\bare{F}$ as a $R_0$-linear combination of $\bare{S_i}$ with $S_i$ ranging over SSYT of shape $\lambda/\mu$ will be referred to as straightening in $E^{\lambda/\mu}$.

This section develops a framework for straightening via rearrangement coefficients in the skew Schur module. We introduce the action of the column Young subgroup $S_{\lambda/\mu}$ on fillings by independent column permutations and use it to define, for two fillings $F, S$ of the same shape and content, the rearrangement coefficient. We establish core structural properties of these coefficients and construct an $R$-linear functional that extracts them from determinantal images. Building on this, we define a new basis of $E^{\lambda/\mu}$ and prove an explicit, non-iterative straightening formula that expresses any filling in this basis with coefficients given by the corresponding rearrangement coefficients.

\subsection{Column permutations and straightening}
We now introduce the column permutation subgroup and its action on fillings. Define the Young subgroup
\[
S_{\lambda/\mu}\;:=\;\mathfrak{S}_{\ell'_1}\times\cdots\times \mathfrak{S}_{\ell'_{\lambda_1}},
\]
whose elements $\underline{\pi}=(\pi_1,\ldots,\pi_{\lambda_1})$ act by permuting entries independently within each column of $\lambda/\mu$. Composition, inversion, and sign are taken componentwise
\[
\underline{\pi}\,\underline{\sigma}=(\pi_1 \sigma_1,\ldots,\pi_{\lambda_1}\sigma_{\lambda_1}),\quad
\underline{\pi}^{-1}=((\pi_1)^{-1},\ldots,(\pi_{\lambda_1})^{-1}), \text{ and } \sign{\underline{\pi}} = \sign{\pi_1}\cdots\sign{\pi_{\lambda_1}},
\]
where we view \(\operatorname{sgn}:\mathfrak{S}_d\!\to\!\{\pm1\}\!\subset\! R_0\). Then $\sign{\underline{\pi}\, \underline{\sigma}}\!=\!\sign{\underline{\pi}} \sign{\underline{\sigma}}$ and $\sign{\underline{\pi}^{-1}}\!=\!\sign{\underline{\pi}}$.

For a filling $F$ and $\underline{\pi}\in S_{\lambda/\mu}$, define the filling $F_{\underline{\pi}}$ by column-wise permutation,
\begin{align*}
F_{\underline{\pi}}[\row_{c,i},c]&=F\big[\row_{c,\pi_c^{-1}(i)},\,c\big]\quad\text{for }1\le i\le \ell'_c,
\end{align*}
where the inverse ensures that the action is a left action. In $E^{\lambda/\mu}$ one has, by the column-alternating property of Definition~\ref{def:universal-property-skew-schur-module}, that
\[
\bare{F}\;=\;\sign{\underline{\pi}}\cdot \bare{F_{\underline{\pi}}}.
\]

Two fillings $F$ and $S$ have the same \emph{row content} if, for every fixed $r \in [\lambda'_1]$ we have the multiset equality
\[
\big\{\,F[r,c] \,\big|\, (r,c)\in \lambda/\mu\,\big\}
=
\big\{\,S[r,c] \,\big|\, (r,c)\in \lambda/\mu\,\big\}.
\]

The \emph{rearrangement subset} of $S_{\lambda/\mu}$ associated to $F,S \in F(\lambda/\mu, z)$ is the set
\begin{equation}\label{def:rearrangement_subset}
S_{\lambda/\mu}(F,S) \;=\; \big\{\, \underline{\pi}\in S_{\lambda/\mu} \ \big|\  F_{\underline{\pi}}\textrm{ and }S \textrm{ have the same row content}\, \big\}.
\end{equation}
Note that this definition is not symmetric with respect to $F$ and $S$.

\begin{definition}\label{def:rearrangement_coefficient}
Let $F, S \in F(\lambda/\mu, z)$. The \emph{rearrangement coefficient} of $F$ with respect to $S$ is
\[
\rcf{F}{S} = \sum_{\underline{\pi}\in S_{\lambda/\mu}(F,S)} \sign{\underline{\pi}}\,.
\]
\end{definition}

\begin{example}
Let $\lambda = (3, 2, 1)$, $\mu = (1,1)$ and $z=(2,1,1)$ with $F,S \in F(\lambda/\mu, z)$ such that
\[
\ytableausetup{mathmode,boxsize=1em,centertableaux}
F=\begin{ytableau}
\none & 2 & 1 \\
\none & 3 \\
1
\end{ytableau}
\qquad \textrm{and} \qquad
S=\begin{ytableau}
\none & 1 & 3 \\
\none & 2 \\
1
\end{ytableau}\, .
\]
Writing permutations in one-line notation, $\underline{\pi} = (1,21,1) \in S_{\lambda/\mu}(F, S)$ since
\[
\ytableausetup{mathmode,boxsize=1em,centertableaux}
F_{\underline{\pi}}=\begin{ytableau}
\none & 3 & 1 \\
\none & 2 \\
1
\end{ytableau}
\]
has the same row content as $S$. Trivially, $\underline{\pi}$ is the only element in $S_{\lambda/\mu}(F, S)$ and hence $\rcf{F}{S}=\sign{\underline{\pi}}\!=\!\sign{1}\sign{21}\sign{1} \!=\! -1$. Conversely, $S_{\lambda/\mu}(S, F)$ is empty and $\rcf{S}{F}=0$.
\end{example}

\begin{definition}\label{def:rowsort&sort}
For $F \in F(\lambda/\mu, z)$, the \emph{row-sorting} $\operatorname{rowsort}(F)$ is obtained by reordering the entries within each row of $F$ so that they are weakly increasing along the row. The \emph{sorting} $\operatorname{sort}(F)$ is obtained by first reordering the entries within each column so that they are weakly increasing downward, and then applying row-sorting, or equivalently,
\[
\operatorname{sort}(F)=\operatorname{rowsort}(\operatorname{colsort}(F)).
\]
\end{definition}

\begin{lemma}\label{lemma:fundRCoeff}
Let $F, T, S \in F(\lambda/\mu, z)$ and let $\underline{\sigma}, \underline{\gamma}, \underline{\sigma}', \underline{\gamma}' \in S_{\lambda/\mu}$. Suppose that
\[
\underline{\sigma}\,\underline{\pi}\, \underline{\gamma}\ \in\ S_{\lambda/\mu}(F,S)\quad\text{for all }\ \underline{\pi}\in S_{\lambda/\mu}(T,S),
\]
and
\[
\underline{\sigma}'\, \underline{\pi}'\, \underline{\gamma}'\ \in\ S_{\lambda/\mu}(T,S)\quad\text{for all }\ \underline{\pi}'\in S_{\lambda/\mu}(F,S).
\]
Then
\[
\rcf{F}{S}=\operatorname{sgn}(\underline{\sigma})\,\operatorname{sgn}(\underline{\gamma})\,\rcf{T}{S}
\;=\;
\operatorname{sgn}(\underline{\sigma}')\,\operatorname{sgn}(\underline{\gamma}')\,\rcf{T}{S}.
\]
\end{lemma}

\begin{proof}
Define $\Phi:S_{\lambda/\mu}(T,S)\to S_{\lambda/\mu}(F,S)$ by $\Phi(\underline{\pi})=\underline{\sigma}\, \underline{\pi}\, \underline{\gamma}$ and $\Psi:S_{\lambda/\mu}(F,S)\to S_{\lambda/\mu}(T,S)$ by $\Psi(\underline{\pi}')=\underline{\sigma}'\, \underline{\pi}'\, \underline{\gamma}'$. By the hypotheses, both maps are well defined and injective, and since the sets are finite, $\Phi$ and $\Psi$ are bijections. Therefore,
\[
\rcf{F}{S}
=\sum_{\underline{\pi}'\in S_{\lambda/\mu}(F,S)} \operatorname{sgn}(\underline{\pi}')
=\sum_{\underline{\pi}\in S_{\lambda/\mu}(T,S)} \operatorname{sgn}\big(\Phi(\underline{\pi})\big)
=\operatorname{sgn}(\underline{\sigma})\,\operatorname{sgn}(\underline{\gamma})\,\rcf{T}{S},
\]
and similarly
\[
\rcf{T}{S}
=\sum_{\underline{\pi}\in S_{\lambda/\mu}(T,S)} \operatorname{sgn}(\underline{\pi})
=\sum_{\underline{\pi}'\in S_{\lambda/\mu}(F,S)} \operatorname{sgn}\big(\Psi(\underline{\pi}')\big)
=\operatorname{sgn}(\underline{\sigma}')\,\operatorname{sgn}(\underline{\gamma}')\,\rcf{F}{S}.
\]
Multiplying the second equality by $\operatorname{sgn}(\underline{\sigma}')\,\operatorname{sgn}(\underline{\gamma}')$ and substituting into the first yields our desired equalities.
\end{proof}

\begin{corollary}\label{corollary:fundRCoeff}
Let $F, T \in F(\lambda/\mu, z)$ and $\underline{\pi}\in S_{\lambda/\mu}$. Then $\rcf{F_{\underline{\pi}}}{T}=\operatorname{sgn}(\underline{\pi})\,\rcf{F}{T}$.
\end{corollary}

\begin{proof}
For any $\underline{\tau}\in S_{\lambda/\mu}(F,T)$, one has
\[
(F_{\underline{\pi}})_{\,\underline{\tau}\,\underline{\pi}^{-1}}
=F_{\,\underline{\tau}\,\underline{\pi}^{-1}\,\underline{\pi}}
=F_{\,\underline{\tau}},
\]
which has the same row content as $T$. Hence $\underline{\tau}\,\underline{\pi}^{-1}\in S_{\lambda/\mu}(F_{\underline{\pi}},T)$.

Conversely, if $\underline{\tau}'\in S_{\lambda/\mu}(F_{\underline{\pi}},T)$, then
\[
F_{\,\underline{\tau}'\,\underline{\pi}}
=(F_{\underline{\pi}})_{\,\underline{\tau}'},
\]
which has the same row content as $T$. Hence $\underline{\tau}'\,\underline{\pi}\in S_{\lambda/\mu}(F,T)$.

Applying Lemma~\ref{lemma:fundRCoeff} with $\underline{\sigma}=\underline{\text{id}}$, $\underline{\gamma}=\underline{\pi}^{-1}$ and $\underline{\sigma}'=\underline{\text{id}}$, $\underline{\gamma}'=\underline{\pi}$ yields
\[
\rcf{F_{\underline{\pi}}}{T}
=\operatorname{sgn}(\underline{\pi}^{-1})\,\rcf{F}{T}
=\operatorname{sgn}(\underline{\pi})\,\rcf{F}{T}. \qedhere
\]
\end{proof}

The following fact will be used repeatedly in the next few results.

\begin{lemma}\label{lemma:reorderIncreases}
Let $F\in F(\lambda/\mu,z)$ with $\bare{F}\neq 0$ in $E^{\lambda/\mu}$, and set $F_c:=\operatorname{colsort}(F)$. For any nontrivial $\underline{\pi}\in S_{\lambda/\mu}$ one has
\[
\operatorname{sort}(F)\ \prec\ \operatorname{rowsort}\bigl((F_c)_{\underline{\pi}}\bigr).
\]
\end{lemma}

\begin{proof}
Since $\bare{F}\neq 0$, no column of $F$ contains duplicate entries, and so each column of $F_c$ is strictly increasing downward. Let $r$ be the smallest row where $(F_c)_{\underline{\pi}}$ and $F_c$ differ, and let $c'$ be the smallest column in that row with $(F_c)_{\underline{\pi}}[r,c']\neq F_c[r,c']$. In column $c'$, the post-permutation value at row $r$ equals some entry that originally lay in row $s$ of that same column. By minimality of $r$, one must have $s>r$. Since $F_c$ is strictly increasing down each column, this implies
\[
(F_c)_{\underline{\pi}}[r,c'] \;=\; F_c[s,c'] \;>\; F_c[r,c'].
\]
For any other column $c$ and the same row $r$, either the value is unchanged or it is likewise replaced by a value from a strictly lower row of that column, hence strictly larger. Therefore, the multiset of entries in row $r$ of $(F_c)_{\underline{\pi}}$ is obtained from the multiset of entries in row $r$ of $F_c$ by replacing a (nonempty) submultiset by strictly larger values, with all other entries unchanged. Rows $1,\dots,r-1$ agree.

Row-sorting replaces each row by its nondecreasing rearrangement. Let $k=\lambda_r-\mu_r$ and write the entries of the $r$th rows of $\operatorname{rowsort}(F_c)$ and $\operatorname{rowsort}\bigl((F_c)_{\underline{\pi}}\bigr)$, respectively, as
\[
a_1\le\cdots\le a_k,
\qquad \text{ and }
b_1\le\cdots\le b_k.
\]
By the replacement of a submultiset by strictly larger values, we have $a_j\le b_j$ for all $j$, with $a_{j_0}<b_{j_0}$ for at least one $j_0$. Since the reading word order is lexicographic by rows and within each row from left to right, rows $1,\dots,r-1$ agree and the first difference occurs in row $r$ at the minimal index where $a_j\neq b_j$, where necessarily $a_j<b_j$. Hence
\[
\operatorname{sort}(F)\ \prec\ \operatorname{rowsort}\bigl((F_c)_{\underline{\pi}}\bigr). \qedhere
\]
\end{proof}

\begin{proposition}\label{prop:rCoeffOrder}
Let $F \in F(\lambda/\mu, z)$ with $\bare{F}\neq 0$ in $E^{\lambda/\mu}$, and set $F_c:=\operatorname{colsort}(F)$. Then $F_c=F_{\underline{\sigma}}$ for some $\underline{\sigma}\in S_{\lambda/\mu}$. For $S\in \operatorname{SSYT}(\lambda/\mu, z)$ the following hold.
\begin{enumerate}
\item[(i)] If $S \prec \operatorname{sort}(F)$, then $\rcf{F}{S}=0$.
\item[(ii)] $\rcf{F_c}{\operatorname{sort}(F)}=1$.
\item[(iii)] $\rcf{F}{\operatorname{sort}(F)}=\operatorname{sgn}(\underline{\sigma})$.
\end{enumerate}
\end{proposition}

\begin{proof}
Since $\bare{F}\neq 0$, no column of $F$ has duplicate entries, hence each column of $F_c$ is strictly increasing downward.

\noindent (i) By Corollary~\ref{corollary:fundRCoeff}, $\rcf{F_c}{S}=\operatorname{sgn}(\underline{\sigma})\,\rcf{F}{S}$ with $F_c=F_{\underline{\sigma}}$. It suffices to show $S\prec\operatorname{sort}(F_c)$ implies $\rcf{F_c}{S}=0$. Suppose there exists $\underline{\pi}\in S_{\lambda/\mu}(F_c,S)$. Then $S=\operatorname{rowsort}((F_c)_{\underline{\pi}})$. If $\underline{\pi}=\underline{\text{id}}$, then $S=\operatorname{rowsort}(F_c)=\operatorname{sort}(F)$, a contradiction. If $\underline{\pi}\neq\underline{\text{id}}$, Lemma~\ref{lemma:reorderIncreases}  gives
\[
\operatorname{sort}(F)\ \prec\ \operatorname{rowsort}\bigl((F_c)_{\underline{\pi}}\bigr)=S,
\]
again a contradiction. Hence $S_{\lambda/\mu}(F_c,S)=\varnothing$ and $\rcf{F_c}{S}=0$, whence $\rcf{F}{S}=0$.

\noindent (ii) The identity $\underline{\text{id}}$ is in $S_{\lambda/\mu}(F_c,\operatorname{sort}(F))$ since $\operatorname{rowsort}(F_c)=\operatorname{sort}(F)$. If $\underline{\pi}\neq\underline{\text{id}}$ and $(F_c)_{\underline{\pi}}$ has the same row content as $\operatorname{sort}(F)$, then
\[
\operatorname{rowsort}\bigl((F_c)_{\underline{\pi}}\bigr)=\operatorname{sort}(F).
\]
By Lemma~\ref{lemma:reorderIncreases}, $\operatorname{sort}(F)\prec \operatorname{rowsort}\bigl((F_c)_{\underline{\pi}}\bigr)$, yielding a contradiction. Thus $S_{\lambda/\mu}(F_c,\operatorname{sort}(F))=\{\underline{\text{id}}\}$ and $\rcf{F_c}{\operatorname{sort}(F)}=1$.

\noindent (iii) Using $F_c=F_{\underline{\sigma}}$, Corollary~\ref{corollary:fundRCoeff}, and part (ii) above gives
\[
\rcf{F}{\operatorname{sort}(F)}=\operatorname{sgn}(\underline{\sigma})\,\rcf{F_c}{\operatorname{sort}(F)}=\operatorname{sgn}(\underline{\sigma}).
\]\end{proof}

\subsection{An R-linear map}
\label{subsec:RLinear}
Fix a skew diagram $\lambda/\mu$, content $z$, and $S\in \operatorname{SSYT}(\lambda/\mu, z)$. Let $Z_{i,j}$ for $1 \leq i \leq \lambda'_1$ and $1 \leq j \leq m$ be a set of indeterminates, and $R[Z]$ the polynomial ring over $R$ in these indeterminates. The goal of this section is to construct an $R$-module homomorphism $\reval{S}:E^{\lambda/\mu} \rightarrow R$ that maps $F \in F(\lambda/\mu, z)$ to $\rcf{F}{S}$. This homomorphism $\reval{S}$ will be defined as the composition of two $R$-module homomorphisms.

In Proposition~\ref{prop:lemma3inFulton} we construct an $R$-module homomorphism $\varphi:E^{\lambda/\mu} \rightarrow R[Z]$ that maps $F \in F(\lambda/\mu, z)$ to $D_F$.

Let  $\mathfrak{C}_{-,S}$ be the map from $R[Z]$ to $R$ that sends $p \in R[Z]$ to the coefficient of the monomial
\[
M_S = \prod_{(r,c) \in \lambda/\mu} Z_{r, S[r,c]} = \displaystyle \prod_{j=1}^{\lambda_1} \prod_{i=1}^{\lambda_j'} Z_{i, S(i,j)},
\]
so that $\mathfrak{C}_{-,S}(p)=0$ if $M_S$ does not appear in $p$. This map is an $R$-module homomorphism.

Finally, we define the map $\reval{S}:E^{\lambda/\mu} \rightarrow R$ as the composition of $\varphi$ and $\mathfrak{C}_{-,S}$, that is,
\[
  \reval{S}:= \mathfrak{C}_{-,S} \circ \varphi.
\]

\begin{proposition}
\label{prop:theRModuleHom}
Fix a skew diagram $\lambda/\mu$, content $z$, and $S\in \operatorname{SSYT}(\lambda/\mu, z)$. The map $\reval{S}$ is an $R$-module homomorphism from $E^{\lambda/\mu}$ to $R$ that maps $F \in F(\lambda/\mu, z)$ to $\rcf{F}{S}$.
\end{proposition}
\begin{proof}
We aim to show that $\mathfrak{C}_{-,S}(D_F) = \rcf{F}{S}$. By definition
\[
D_F = \prod_{c=1}^{\lambda_1} D_{F, c} = \prod_{c=1}^{\lambda_1} \det(M_{F,c}).
\]
Using the Leibniz formula for the determinant of each $M_{F,c}$,
\begin{align*}
D_F &= \prod_{c=1}^{\lambda_1} \left( \sum_{\sigma_c \in \mathfrak{S}_{\ell_c'}} \sign{\sigma_c} \prod_{i=1}^{\ell_c'} Z_{\row_{c,i}, F[\row_{c, \sigma_c(i)}, c]} \right).
\end{align*}
For each column $c$, we replace the permutation $\sigma_c$ with its inverse $\pi_c := \sigma_c^{-1}$. Since $\sign{\sigma_c} = \sign{\pi_c}$ and the map is a bijection on $\mathfrak{S}_{\ell_c'}$, the sum remains the same. Hence,
\[
D_F = \prod_{c=1}^{\lambda_1} \left( \sum_{\pi_c \in \mathfrak{S}_{\ell_c'}} \sign{\pi_c} \prod_{i=1}^{\ell_c'} Z_{\row_{c,i}, F[\row_{c, \pi_c^{-1}(i)}, c]} \right).
\]
By the definition of $\underline{\pi}=(\pi_1, \dots, \pi_{\lambda_1}) \in S_{\lambda/\mu}$ and its action on fillings,
\begin{align*}
D_F
&= \sum_{\underline{\pi} \in S_{\lambda/\mu}} \operatorname{sgn}(\underline{\pi})\,
\prod_{c=1}^{\lambda_1} \prod_{i=1}^{\ell_c'} Z_{\row_{c,i},\, F[\row_{c, \pi_c^{-1}(i)},\, c]} \\
&= \sum_{\underline{\pi} \in S_{\lambda/\mu}} \operatorname{sgn}(\underline{\pi})\,
\prod_{c=1}^{\lambda_1} \prod_{i=1}^{\ell_c'} Z_{\row_{c,i},\, F_{\underline{\pi}}[\row_{c,i},\, c]}\,.
\end{align*}

This can be rewritten by swapping the order of the products to be over all cells $(r,c)$ in the diagram of shape $\lambda/\mu$
\[
D_F = \sum_{\underline{\pi} \in S_{\lambda/\mu}} \sign{\underline{\pi}} \prod_{(r,c) \in \lambda/\mu} Z_{r, F_{\underline{\pi}}[r,c]}.
\]
The map $\mathfrak{C}_{-,S}$ extracts the coefficient of $M_S$.
Hence, a term corresponding to $\underline{\pi}$ contributes to the coefficient of $M_S$ if and only if the monomial part is identical to $M_S$. That is, if
\[
\prod_{(r,c) \in \lambda/\mu} Z_{r, F_{\underline{\pi}}[r,c]} = \prod_{(r,c) \in \lambda/\mu} Z_{r, S[r,c]}.
\]
This equality holds if and only if the multiset of values in each row of $F_{\underline{\pi}}$ is the same as the multiset of values in the corresponding row of $S$. This is precisely the condition for $\underline{\pi} \in S_{\lambda/\mu}(F, S)$ from \eqref{def:rearrangement_subset}.

Therefore, when we apply $\mathfrak{C}_{-,S}$ to $D_F$, we sum the signs of only those permutations $\underline{\pi}$ that are in $S_{\lambda/\mu}(F,S)$ to get
\[
\mathfrak{C}_{-,S}(D_F) = \sum_{\underline{\pi} \in S_{\lambda/\mu}(F,S)} \sign{\underline{\pi}} = \rcf{F}{S}. \qedhere
\]
\end{proof}

\subsection{The D-basis}

Fix the reading-word order $\prec$ on $\operatorname{SSYT}(\lambda/\mu, z)$ and label its elements so that
\[
S_{n} \prec S_{n-1} \prec \cdots \prec S_2 \prec S_1,
\]
where $n=|\operatorname{SSYT}(\lambda/\mu, z)|$. Let $\{\bare{S_1},\ldots,\bare{S_n}\}$ denote the corresponding basis of $M_{R,z}/Q_z$.

\begin{definition}[D-basis]\label{def:d_basis}
Define vectors $\{\bare{D_1},\ldots,\bare{D_n}\}\subset M_{R,z}/Q_z$ recursively by
\[
\bare{D_1}:=\bare{S_1},\qquad
\bare{D_i}:=\bare{S_i}-\sum_{j=1}^{i-1}\rcf{S_i}{S_j}\,\bare{S_j}\quad\text{for }i>1.
\]
The set $\{\bare{D_i}\}$ is a basis of $M_{R,z}/Q_z$ since the transition matrix from $\{\bare{S_i}\}$ is unitriangular.
\end{definition}

\section{A Non-iterative Straightening Algorithm for Skew Fillings} \label{sec:noniterativestraighten}
This section proves the main result of this paper, namely that every filling of fixed shape and content admits a non-iterative straightening expansion in the D-basis, with coefficients given explicitly by rearrangement coefficients. The argument proceeds in two steps. First, using the ordering of SSYT and the definition of the D-basis, we establish a unitriangular change of basis showing that each $\bare{S_i}$ expands as $\sum_j \rcf{S_i}{S_j}\,\bare{D_j}$. Second, for an arbitrary filling $\bare{F}$, we combine the $R$-linear evaluation maps $\reval{S}$ with the vanishing and normalization properties from Proposition~\ref{prop:rCoeffOrder} to obtain the straightening formula
\[
\bare{F}\;=\;\sum_{S_j\in \operatorname{SSYT}(\lambda/\mu, z)} \rcf{F}{S_j}\,\bare{D_j},
\]
which is non-iterative and computed directly from the rearrangement coefficients.
\begin{lemma}
\label{lemma:mainTheorem1Standard}
Fix a skew diagram $\lambda/\mu$ and content $z$. Let $S_i \in \operatorname{SSYT}(\lambda/\mu, z)$. Then
\[
\bare{S_i}
=\sum_{S_j \in \operatorname{SSYT}(\lambda/\mu, z)} \rcf{S_i}{S_j}\,\bare{D_{j}}.
\]
\end{lemma}
\begin{proof}
For each $j > i$ we have $S_j \prec S_i = \sorting{S_i}$ and thus $\rcf{S_i}{S_j}=0$ by Proposition~\ref{prop:rCoeffOrder}(i). This, combined with Proposition~\ref{prop:rCoeffOrder}(ii), implies
\begin{align*}
\sum_{S_j \in \operatorname{SSYT}(\lambda/\mu, z)} \rcf{S_i}{S_j}\,\bare{D_{j}}
&= \bare{D_{i}} + \sum_{\substack{S_j \in \operatorname{SSYT}(\lambda/\mu, z) \\ j < i}} \rcf{S_{i}}{S_j}\,\bare{D_{j}} \\
&= \Bigl(\bare{S_i} - \sum_{\substack{S_j \in \operatorname{SSYT}(\lambda/\mu, z) \\ j < i}} \rcf{S_{i}}{S_j}\,\bare{D_{j}}\Bigr)
   + \sum_{\substack{S_j \in \operatorname{SSYT}(\lambda/\mu, z) \\ j < i}} \rcf{S_{i}}{S_j}\,\bare{D_{j}} \\
&= \bare{S_i}. \qedhere
\end{align*}
\end{proof}

\begin{proposition}
\label{prop:mainStraighten}
Let $F \in F(\lambda/\mu, z)$, with $\bare{F}=\sum\limits_ {S_i \in \operatorname{SSYT}(\lambda/\mu, z)}  a_i \bare{S_i}$ in $M_{R, z}/Q_z$ and $a_i \in R$. Then
\[
\rcf{F}{S_j} = \displaystyle \sum_{S_i \in \operatorname{SSYT}(\lambda/\mu, z)} a_i \rcf{S_i}{S_j}
\]
for each $S_j \in \operatorname{SSYT}(\lambda/\mu, z)$.
\end{proposition}
\begin{proof}
By Proposition~\ref{prop:theRModuleHom}, $\reval{S_j}$ is an $R$-module homomorphism and applying it to both sides of $\bare{F}=\sum_ {S_i \in \operatorname{SSYT}(\lambda/\mu, z)}  a_i \bare{S_i}$ yields the desired equality for each $S_j \in \operatorname{SSYT}(\lambda/\mu, z)$.
\end{proof}

Applying Proposition~\ref{prop:mainStraighten} yields the non-iterative straightening formula. This is our first main result.
\thmA*
  \begin{proof}
    Suppose that $\bare{F} = \sum\limits_{S_i \in \operatorname{SSYT}(\lambda/\mu, z)} a_i \bare{S_i}$ in $M_{R, z}/Q_z$.
  By Proposition \ref{prop:mainStraighten}
  \[
  \rcf{F}{S_j} = \displaystyle \sum_{S_i \in \operatorname{SSYT}(\lambda/\mu, z)} a_i  \rcf{S_i}{S_j},
  \]
  for each $S_j \in \operatorname{SSYT}(\lambda/\mu, z)$. Therefore
  \begin{align*}
  0 &=\displaystyle \sum_{S_j \in \operatorname{SSYT}(\lambda/\mu, z)} \left( \rcf{F}{S_j} - \displaystyle \sum_{S_i \in \operatorname{SSYT}(\lambda/\mu, z)} a_i \rcf{S_i}{S_j} \right) \bare{D_{j}} \\
  &=\displaystyle \sum_{S_j \in \operatorname{SSYT}(\lambda/\mu, z)} \rcf{F}{S_j}  \bare{D_{j}} - \displaystyle \sum_{S_j \in \operatorname{SSYT}(\lambda/\mu, z)} \sum_{S_i \in \operatorname{SSYT}(\lambda/\mu, z)} a_i \rcf{S_i}{S_j} \bare{D_{j}} \\
  &= \displaystyle \sum_{S_j \in \operatorname{SSYT}(\lambda/\mu, z)} \rcf{F}{S_j} \bare{D_{j}} - \displaystyle \sum_{S_i \in \operatorname{SSYT}(\lambda/\mu, z)} a_i \sum_{S_j \in \operatorname{SSYT}(\lambda/\mu, z)} \rcf{S_i}{S_j} \bare{D_{j}}.
  \end{align*}
  Finally, applying Lemma~\ref{lemma:mainTheorem1Standard} to this equation yields
  \[
  \displaystyle \sum_{S_j \in \operatorname{SSYT}(\lambda/\mu, z)} \rcf{F}{S_j} \bare{D_{j}} - \displaystyle \sum_{S_i \in \operatorname{SSYT}(\lambda/\mu, z)} a_i \bare{S_i} = 0,
  \]
  and hence
  \[
  \bare{F} = \displaystyle \sum_{S_j \in \operatorname{SSYT}(\lambda/\mu, z)} \rcf{F}{S_j}  \bare{D_{j}}.\qedhere
  \]
  \end{proof}

  \begin{example}
  Fix a skew shape $\lambda / \mu$ with $\lambda=(3,2)$ and $\mu=(1)$, and content $z=(2,1,1)$. The three SSYT in $\operatorname{SSYT}(\lambda/\mu, z)$, ordered by $\prec$, are
  \[
  S_1=\begin{ytableau}
  \none & 1 & 3 \\
  1 & 2
  \end{ytableau},
  \qquad \qquad
  S_2=\begin{ytableau}
  \none & 1 & 2 \\
  1 & 3
  \end{ytableau},
  \qquad \qquad
  S_3=\begin{ytableau}
  \none & 1 & 1 \\
  2 & 3
  \end{ytableau}.
  \]

  The $D$-basis elements are
\[
\begin{aligned}
\bare{D_1}&=\bare{S_1},\\
\bare{D_2}&=\bare{S_2},\\
\bare{D_3}&=\bare{S_3}+\bare{D_1}=\bare{S_3}+\bare{S_1}.
\end{aligned}
\]
  Let $F \in F(\lambda/\mu, z)$ with
  \[
  F=\begin{ytableau}
  \none & 2 & 1 \\
  3 & 1
  \end{ytableau}.
  \]
  Since $\rcf{F}{S_1} = 0$, $\rcf{F}{S_2} = 1$, and $\rcf{F}{S_3} = -1$, Theorem \ref{thm:manResult1} gives
  \[
\begin{aligned}
\bare{F}&=\bare{D_2}-\bare{D_3}
= -\,\bare{S_1}+\bare{S_2}-\bare{S_3}.
\end{aligned}
\]
  \end{example}

\begin{lemma}\label{lemma:sortingInStd}
  Fix a skew diagram $\lambda/\mu$, content $z$, and $F\in F(\lambda/\mu, z)$  such that every column contains distinct entries. Then $\sorting{F}$ is a SSYT.
  \end{lemma}

\begin{proof}
The argument of \cite[Prop.~4.1]{Wil10} applies mutatis mutandis; replace the partition shape by the skew shape $\lambda/\mu$ and require rows to be weakly, rather than strictly, increasing.
\end{proof}

The following corollary sharpens the straightening formula in two ways. It reduces the computational workload by restricting the sum to $j\le k$ when $\sorting{F}=S_k$, and it isolates a canonical leading term, the contribution at $S_k$, with respect to the fixed reading-word order.

\begin{corollary}\label{corollary:mainCorollaryStr1}
Fix a skew diagram $\lambda/\mu$, content $z$, and $F\in F(\lambda/\mu, z)$ with $\sorting{F}=S_k \in \operatorname{SSYT}(\lambda/\mu, z)$. Then
\begin{center}
$\bare{F} = \displaystyle \sum_{\substack{S_j \in \operatorname{SSYT}(\lambda/\mu, z) \\ j \leq k}} \rcf{F}{S_j} \bare{D_{j}}$.
\end{center}
\end{corollary}
\begin{proof}
If $F$ has duplicated entries in each column, then $\bare{F} = 0\in M_{R, z}/Q_z$ and the result is trivial. If $F$ has no duplicated entries, then $\sorting{F} \in \operatorname{SSYT}(\lambda/\mu, z)$ by Lemma \ref{lemma:sortingInStd}. Suppose that $\sorting{F}=S_k \in \operatorname{SSYT}(\lambda/\mu, z)$. By Proposition \ref{prop:rCoeffOrder}(i), we have $\rcf{S_i}{\sorting{F}} =0$ for all $S_i \in \operatorname{SSYT}(\lambda/\mu, z)$ such that $i < k$. Therefore,
\[
\bare{F} = \displaystyle \sum_{\substack{S_j \in \operatorname{SSYT}(\lambda/\mu, z) \\ j \leq k}} \rcf{F}{S_j} \bare{D_{j}}. \qedhere
\]
\end{proof}

\section{Orthogonality for Skew Schur Modules} \label{sec:ortho}
This section intertwines straightening combinatorics with geometry. The straightening rule is revealed as an orthogonal expansion. Endowing $E^{\lambda/\mu}$ with a natural sesquilinear form, the D-basis arises as the Gram-Schmidt orthogonalization of the SSYT basis with respect to the fixed reading-word order, and the rearrangement coefficients become the orthogonal coordinates of $\bare{F}$. In this light, the main theorem is a projection statement; $\bare{F}$ decomposes into its orthogonal components along $\{\bare{D_j}\}$, which both explains the triangular support and isolates a canonical leading term. The resulting picture is computationally efficient and structurally natural, yielding a canonical inner-product framework in which the combinatorial basis is geometrically distinguished.

Throughout this section let $R$ be a commutative ring equipped with an involutive automorphism ${}^{*}$ such that $R^{*}=\{\,r\in R:\,r^{*}=r\,\}$ is an ordered ring, $r\,r^{*}\ge 0$ for all $r\in R$, and the involution is proper in the sense that $a\,a^{*}=0$ implies $a=0$.

\begin{definition}\label{def:inner_product_r_module}
An inner product $R$-module is an $R$-module $M$ together with a map $\langle\cdot,\cdot\rangle:M\times M\to R$ such that for all $u,v,w\in M$ and $r\in R$:
\begin{enumerate}
\item[(i)] (Conjugate symmetry) $\langle u,v\rangle=\langle v,u\rangle^{*}$.
\item[(ii)] (Sesquilinearity) $\langle u+v,w\rangle=\langle u,w\rangle+\langle v,w\rangle$, $\ \langle ru,v\rangle=r\,\langle u,v\rangle$, and $\ \langle u,rv\rangle=r^{*}\,\langle u,v\rangle$.
\item[(iii)] (Positive-definiteness) $\langle u,u\rangle\ge 0$ in $R^{*}$, with $\langle u,u\rangle=0$ if and only if $u=0$.
\end{enumerate}
\end{definition}

For the remainder of this section, fix a skew shape $\lambda/\mu$ and an SSYT basis $\mathcal{S}_{\lambda/\mu}=\{\bare{S_1},\ldots,\bare{S_n}\}$ of $E^{\lambda/\mu}$ ordered $S_n\prec\cdots\prec S_1$. We construct the D-basis contentwise. For each content $z$, we build D-vectors from the SSYT of content $z$. When we assemble all contents into a single global SSYT order $\{S_1,\ldots,S_n\}$, we relabel the D-vectors to match this order, writing the associated vectors as $\{\bare{D_1},\ldots,\bare{D_n}\}$ so that $S_k$ corresponds to $\bare{D_k}$.

\begin{definition}\label{def:sesquilinear_form_local}
Define
\[
\langle \bare{S_i},\,\bare{D_{j}}\rangle \;:=\; \mathcal{R}_{S_i,S_j}.
\]
 Extend this rule sesquilinearly in the first variable and conjugate-linearly in the second, that is, for all $u=\sum_i r_i\bare{S_i}$ and $v=\sum_j s_j\bare{D_{j}}$ with $r_i,s_j\in R$,
\[
\langle u,\,v\rangle \;=\; \sum_{i,j} r_i\,s_j^{*}\,\langle \bare{S_i},\,\bare{D_{j}}\rangle.
\]
By convention, $\mathcal{R}_{F,S}=0$ if $F$ and $S$ have different contents.
\end{definition}

We first prove that the $D$-basis is orthonormal for $\langle\cdot,\cdot\rangle$, namely that $\langle \bare{D_i},\,\bare{D_j}\rangle=\delta_{ij}$ for all $i,j$.
From this we deduce that $\langle\cdot,\cdot\rangle$ is conjugate symmetric and positive definite, which will prove that $E^{\lambda/\mu}$ equipped with $\langle\cdot,\cdot\rangle$ is an inner product $R$-module.

\begin{lemma}\label{lemma:initial_orthogonality}
If $S_i \in \mathcal{S}_{\lambda/\mu}$, we have $\langle \bare{S_i}, \bare{D_{j}} \rangle = 0$ for $i < j$.
\end{lemma}
\begin{proof}
If $S_i$ and $S_j$ have different contents, then $\langle \bare{S_i},\bare{D_j}\rangle=0$ by convention. Otherwise, $i<j$ implies $S_j\prec S_i=\sorting{S_i}$, so $\rcf{S_i}{S_j}=0$ by Proposition~\ref{prop:rCoeffOrder}(i). By Definition~\ref{def:sesquilinear_form_local}, $\langle \bare{S_i},\bare{D_j}\rangle=\rcf{S_i}{S_j}=0$.
\end{proof}

\begin{lemma}\label{lemma:orthogonality_less}
For all $i<j$, $\langle \bare{D_i}, \bare{D_j} \rangle = 0$.
\end{lemma}
\begin{proof}
Proceed by strong induction on $i$. For $i=1$ and any $j>1$,
\[
\langle \bare{D_1}, \bare{D_j} \rangle
= \langle \bare{S_1}, \bare{D_j} \rangle
= \rcf{S_1}{S_j}
= 0,
\]
by Lemma~\ref{lemma:initial_orthogonality}. Fix $i\ge 1$ and assume $\langle \bare{D_k}, \bare{D_m} \rangle=0$ for all $k<i<m$. For $j>i$, Definition~\ref{def:d_basis} gives
\[
\langle \bare{D_i}, \bare{D_j} \rangle
= \Bigl\langle \bare{S_i}-\sum_{k<i}\rcf{S_i}{S_k}\,\bare{D_k},\,\bare{D_j}\Bigr\rangle
= \langle \bare{S_i}, \bare{D_j} \rangle - \sum_{k<i}\rcf{S_i}{S_k}\,\langle \bare{D_k}, \bare{D_j} \rangle.
\]
The first term is $\rcf{S_i}{S_j}=0$ by Lemma~\ref{lemma:initial_orthogonality}, and each term in the sum vanishes by the induction hypothesis (since $k<i<j$). Hence $\langle \bare{D_i}, \bare{D_j} \rangle=0$.
\end{proof}

\begin{lemma}\label{lemma:norm_property}
For all $i$, $\langle \bare{D_i}, \bare{D_i} \rangle = 1$.
\end{lemma}
\begin{proof}
From the definition of $D_i$,
\[
\langle \bare{D_i}, \bare{D_i} \rangle
= \Bigl\langle \bare{S_i}-\sum_{k<i}\rcf{S_i}{S_k}\,\bare{D_k},\,\bare{D_i}\Bigr\rangle
= \langle \bare{S_i}, \bare{D_i} \rangle - \sum_{k<i}\rcf{S_i}{S_k}\,\langle \bare{D_k}, \bare{D_i} \rangle.
\]
By Lemma~\ref{lemma:orthogonality_less}, $\langle \bare{D_k}, \bare{D_i} \rangle=0$ for $k<i$, so the sum vanishes. Finally, $\langle \bare{S_i}, \bare{D_i} \rangle=\rcf{S_i}{S_i}=1$ by Proposition~\ref{prop:rCoeffOrder}(ii), whence $\langle \bare{D_i}, \bare{D_i} \rangle=1$.
\end{proof}

\begin{lemma}\label{lemma:orthogonality_greater}
For all $i>j$, $\langle \bare{D_i}, \bare{D_j} \rangle = 0$.
\end{lemma}
\begin{proof}
We proceed by strong induction on $i$. For the base case $i=2$,
\[
\langle \bare{D_2}, \bare{D_1} \rangle
= \Big\langle \bare{S_2}-\rcf{S_2}{S_1}\,\bare{D_1},\,\bare{D_1}\Big\rangle
= \rcf{S_2}{S_1}-\rcf{S_2}{S_1}\,\langle \bare{D_1},\bare{D_1}\rangle
= 0
\]
by Lemma~\ref{lemma:norm_property}. Now fix $i\ge2$ and assume as the strong induction hypothesis that $\langle \bare{D_k},\bare{D_m}\rangle=0$ for all $k<i$ and $m<k$. For any $j<i$,
\begin{equation}\label{eq:orthgreater}
\langle \bare{D_i}, \bare{D_j} \rangle
= \Big\langle \bare{S_i}-\sum_{k<i}\rcf{S_i}{S_k}\,\bare{D_k},\,\bare{D_j}\Big\rangle
= \rcf{S_i}{S_j}-\sum_{k<i}\rcf{S_i}{S_k}\,\langle \bare{D_k},\bare{D_j}\rangle.
\end{equation}
We split the sum at $k=j$:
\[
\sum_{k<i}\rcf{S_i}{S_k}\,\langle \bare{D_k},\bare{D_j}\rangle
= \sum_{k<j}\rcf{S_i}{S_k}\,\langle \bare{D_k},\bare{D_j}\rangle
\;+\; \rcf{S_i}{S_j}\,\langle \bare{D_j},\bare{D_j}\rangle
\;+\; \sum_{j<k<i}\rcf{S_i}{S_k}\,\langle \bare{D_k},\bare{D_j}\rangle.
\]
By Lemma~\ref{lemma:orthogonality_less}, the first sum is $0$. By Lemma~\ref{lemma:norm_property}, the middle term equals $\rcf{S_i}{S_j}\cdot 1=\rcf{S_i}{S_j}$. By the induction hypothesis, the last sum is $0$. Hence the entire sum equals $\rcf{S_i}{S_j}$, which cancels the leading $\rcf{S_i}{S_j}$ on the right-hand side of \eqref{eq:orthgreater}, and therefore $\langle \bare{D_i}, \bare{D_j} \rangle=0$.
\end{proof}

\begin{corollary}\label{cor:orthonormal_basis}
The D-basis is orthonormal with respect to $\langle\cdot,\cdot\rangle$, that is, $\langle \bare{D_i}, \bare{D_j} \rangle = \delta_{ij}$, where $\delta_{ij}$ is the Kronecker delta.
\end{corollary}
\begin{proof}
Lemmas~\ref{lemma:orthogonality_less}, \ref{lemma:norm_property}, and \ref{lemma:orthogonality_greater} give the claim.
\end{proof}

With orthonormality established, we now verify that $E^{\lambda / \mu}$ equipped with $\langle\cdot,\cdot\rangle$ is an inner product $R$-module.

\begin{proposition}\label{prop:general_symmetry}
The sesquilinear form $\langle \cdot, \cdot \rangle$ on $E^{\lambda/\mu}$ is conjugate symmetric.
\end{proposition}
\begin{proof}
Let $u,v\in E^{\lambda/\mu}$, and write (using that $\{\bare{D_1},\ldots,\bare{D_n}\}$ is a basis)
\[
u=\sum_{i=1}^{n} a_i\,\bare{D_i},\qquad v=\sum_{j=1}^{n} b_j\,\bare{D_j},\qquad a_i,b_j\in R.
\]
By sesquilinearity,
\[
\langle u,v\rangle=\Big\langle \sum_i a_i\,\bare{D_i},\,\sum_j b_j\,\bare{D_j}\Big\rangle
=\sum_{i,j} a_i\,b_j^{*}\,\langle \bare{D_i},\bare{D_j}\rangle.
\]
At this point we invoke Corollary~\ref{cor:orthonormal_basis} to use $\langle \bare{D_i},\bare{D_j}\rangle=\delta_{ij}$, which yields
\[
\langle u,v\rangle=\sum_{k=1}^{n} a_k\,b_k^{*}.
\]
Similarly,
\[
\langle v,u\rangle=\sum_{k=1}^{n} b_k\,a_k^{*}.
\]
Now using the involutive automorphism $^{*}$ and commutativity of $R$,
\[
(\langle v,u\rangle)^{*}
=\Big(\sum_{k} b_k\,a_k^{*}\Big)^{*}
=\sum_{k} (b_k\,a_k^{*})^{*}
=\sum_{k} b_k^{*}\,(a_k^{*})^{*}
=\sum_{k} a_k\,b_k^{*}
=\langle u,v\rangle.
\]
Hence $\langle u,v\rangle=(\langle v,u\rangle)^{*}$, as claimed.
\end{proof}

\begin{corollary}\label{cor:positive_semidefinite}
For all $u\in E^{\lambda/\mu}$, one has $\langle u,u\rangle\in R^{*}$ with $\langle u,u\rangle\ge 0$. Moreover, $\langle u,u\rangle=0$ if and only if $u=0$.
\end{corollary}
\begin{proof}
Write $u=\sum_{k=1}^{n} a_k\,\bare{D_k}$. By sesquilinearity and Corollary~\ref{cor:orthonormal_basis},
\[
\langle u,u\rangle=\sum_{k=1}^{n} a_k\,a_k^{*}\in R^{*}.
\]
By the standing assumptions on $R$, each $a_k a_k^{*}\ge 0$, hence the sum is nonnegative in the ordered ring $R^{*}$. If $\langle u,u\rangle=0$, then by order properties of $R^{*}$ and $a_k a_k^{*}\ge 0$ for each $k$, it follows that $a_k a_k^{*}=0$ for all $k$, hence $a_k=0$ for all $k$ by the properness assumption. Therefore $u=0$. The converse is immediate.
\end{proof}

\begin{corollary}\label{cor:inner_product_space_corollary}
The module $E^{\lambda/\mu}$ is an inner product $R$-module with respect to $\langle\cdot,\cdot\rangle$. Moreover, if $R=\mathbb{R}$ or $\mathbb{C}$, then $E^{\lambda/\mu}$ is an inner product space.
\end{corollary}
\begin{proof}
Sesquilinearity holds by definition, conjugate symmetry is Proposition~\ref{prop:general_symmetry}, and positive-definiteness is Corollary~\ref{cor:positive_semidefinite}. Thus $E^{\lambda/\mu}$ equipped with $\langle\cdot,\cdot\rangle$ is an inner poduct $R$-module; over $\mathbb{R}$ or $\mathbb{C}$ this is the usual notion of inner product space.
\end{proof}

\begin{definition}
Let $M$ equipped with $\langle\cdot,\cdot\rangle$ be an inner product $R$-module, and let $\{v_1,\dots,v_n\}$ be linearly independent in $M$. One performs the \emph{Gram–Schmidt orthogonalization} process as follows. Set $w_1:=v_1$. For $k\ge2$, assume $w_1,\dots,w_{k-1}$ have been constructed and are pairwise orthogonal. If there exist coefficients $c_{k1},\dots,c_{k \, k-1}\in R$ solving the linear equations
\begin{align*}
c_{kj}\,\langle w_j,w_j\rangle \;=\; \langle v_k,w_j\rangle \qquad (1\le j<k),
\end{align*}
define
\begin{align*}
w_k \;:=\; v_k \;-\; \sum_{j<k} c_{kj}\,w_j.
\end{align*}
Then $w_k$ is orthogonal to $w_1,\dots,w_{k-1}$ and $\operatorname{span}_R\{w_1,\dots,w_k\}=\operatorname{span}_R\{v_1,\dots,v_k\}$. The normalization step is separate. If each $\langle w_j,w_j\rangle$ is a unit in $R$ that admits a square root, one may set $e_j:=\langle w_j,w_j\rangle^{-1/2} w_j$ to obtain an orthonormal sequence $(e_1,\dots,e_n)$. Over $R=\mathbb{R}$ or $\mathbb{C}$ the equations always have solutions $c_{kj}=\langle v_k,w_j\rangle/\langle w_j,w_j\rangle$ and square roots exist, so the process terminates and yields an orthonormal basis in finitely many steps.
\end{definition}

\begin{remark}
Over a general $R$ satisfying our standing assumptions, the algorithm produces an orthogonal (respectively, orthonormal) sequence if and only if, at each step $k$, the linear equations $c_{kj}\,\langle w_j,w_j\rangle=\langle v_k,w_j\rangle$ are solvable for all $j<k$ (respectively, solvable and each $\langle w_j,w_j\rangle$ is a unit admitting a square root in $R$). In particular, if at every stage $\langle w_j,w_j\rangle=1$, then $c_{kj}=\langle v_k,w_j\rangle$ and the output is automatically orthonormal, with no division or square roots at any step.
\end{remark}

\begin{corollary}\label{cor:gs-for-D-basis}
In $E^{\lambda/\mu}$ with the fixed SSYT order, taking $v_i:=\bare{S_i}$, the Gram--Schmidt orthogonalization exists at every step and yields $\bare{D_i}$, and each vector is already normalized,
\begin{align*}
\bare{D_i} \;=\; \bare{S_i} \;-\; \sum_{k<i} \langle \bare{S_i},\bare{D_k}\rangle \,\bare{D_k},
\qquad \langle \bare{D_i},\bare{D_i}\rangle \;=\; 1,
\end{align*}
so the D-basis is the result of the Gram--Schmidt process for any $R$ satisfying our standing assumptions.
\end{corollary}

We now have all ingedients needed to prove our second main result.

\begin{proof}[Proof of Theorem~\ref{thm:B}.]
Corollary~\ref{cor:inner_product_space_corollary} and Corollary~\ref{cor:gs-for-D-basis} immediately imply our desired result. \qedhere
\end{proof}

\printbibliography
\end{document}